\documentclass[a4paper, oneside]{amsart}

\title[The APS index and domain-wall fermion Dirac operators]{The Atiyah-Patodi-Singer index and domain-wall fermion Dirac operators}

\author[H. Fukaya]{Hidenori Fukaya}
\author[M. Furuta]{Mikio Furuta}
\author[S. Matsuo]{Shinichiroh Matsuo}\thanks{Corresponding author: Shinichiroh Matsuo, shinichiroh@math.nagoya-u.ac.jp}
\author[T. Onogi]{Tetsuya Onogi}
\author[S. Yamaguchi]{Satoshi Yamaguchi}
\author[M. Yamashita]{Mayuko Yamashita}

\address[HF, TO, and SY]{%
  Department of Physics \\
  Osaka University \\
  Osaka, Japan
}
\email{hfukaya@het.phys.sci.osaka-u.ac.jp}
\email{onogi@het.phys.sci.osaka-u.ac.jp}
\email{yamaguch@het.phys.sci.osaka-u.ac.jp}

\address[MF]{%
  Graduate School of Mathematical Sciences \\
  The University of Tokyo \\
  Tokyo, Japan
}
\email{furuta@ms.u-tokyo.ac.jp}

\address[SM]{%
  Graduate School of Mathematics \\
  Nagoya University \\
  Nagoya, Japan
}
\email{shinichiroh@math.nagoya-u.ac.jp}
\urladdr{https://www.math.nagoya-u.ac.jp/~shinichiroh/}

\address[MY]{
  Research Institute for Mathematical Sciences \\
  Kyoto University \\
  Kyoto, Japan
}
\email{mayuko@kurims.kyoto-u.ac.jp}

\usepackage[a4paper]{geometry}
\usepackage{microtype}

\usepackage[utf8]{inputenc}
\usepackage[T1]{fontenc}

\usepackage[p,osf]{cochineal}
\usepackage[scale=.95,type1]{cabin}
\usepackage[cochineal,vvarbb]{newtxmath}
\usepackage[cal=boondoxo]{mathalfa}

\usepackage{mathtools}
\mathtoolsset{showonlyrefs=true}

\theoremstyle{plain}
\newtheorem{theorem}{Theorem}
\newtheorem{proposition}[theorem]{Proposition}
\newtheorem{lemma}[theorem]{Lemma}

\theoremstyle{remark}
\newtheorem{remark}[theorem]{Remark}

\usepackage{slashed} % for Dirac operators

\usepackage{tikz}
\usetikzlibrary{babel}
\usetikzlibrary{decorations.pathreplacing}

\newcommand{\superemph}[1]{\textcolor{red}{#1}}

\usepackage[bookmarksnumbered,hidelinks]{hyperref}
\usepackage[msc-links,non-compressed-cites]{amsrefs}

\DeclarePairedDelimiterX{\innerProduct}[2]{(}{)}{#1,#2}
\DeclarePairedDelimiterX{\bracket}[2]{\langle}{\rangle}{#1,#2}
\DeclarePairedDelimiter{\abs}{\lvert}{\rvert}
\DeclarePairedDelimiter{\norm}{\lVert}{\rVert}

\newcommand{\restrictedTo}[2]{\left. #1 \right|_{#2}}

\newcommand{\Z}{\mathbb{Z}}
\newcommand{\R}{\mathbb{R}}
\newcommand{\C}{\mathbb{C}}

\DeclareMathOperator{\Hom}{Hom}
\DeclareMathOperator{\Endo}{End}
\DeclareMathOperator{\Image}{Im}
\DeclareMathOperator{\Ker}{Ker}
\DeclareMathOperator{\id}{id}
\DeclareMathOperator{\tr}{tr}
\newcommand{\transpose}[1]{#1^T}

\newcommand{\Spin}{\mathrm{Spin}}

\DeclareMathOperator{\realpart}{Re}
\DeclareMathOperator{\sign}{sign}

\DeclareMathOperator{\supp}{supp}

\DeclareMathOperator{\spectrum}{Spec}

\newcommand{\gradingOperator}{\Gamma}

\DeclareMathOperator{\ASindex}{Ind}
\DeclareMathOperator{\APSindex}{Ind_{APS}}

\DeclareMathOperator{\signfunction}{sgn}
\newcommand{\kappaAS}{\widehat{\kappa}_{\mathrm{AS}}}
\newcommand{\kappaAPS}{\widehat{\kappa}_{\mathrm{APS}}}
\newcommand{\kappaCorner}{\widehat{\kappa}_{\llcorner }}

\newcommand{\DiracOperatorOnCylinderAS}{\widehat{D}}
\newcommand{\DiracOperatorOnCylinderAPS}{\widehat{\mathcal{D}}}

\newcommand{\kappaAPSsmooth}{\kappaAPS^{\mathrm{sm}}}
\newcommand{\smoothDiracOperatorOnCylinderAPS}{\DiracOperatorOnCylinderAPS^{\mathrm{sm}}}

\newcommand{\DiracOperatorOnZ}{L}

\newcommand{\DiracOperatorOnY}{A}
\newcommand{\firstPositiveEigenvalueOfY}{\lambda_{\DiracOperatorOnY}}

\newcommand{\XPlusHat}{X_{\mathrm{cyl}}}
\newcommand{\DiracOperatorOnXPlusHat}{D_{\mathrm{cyl}}}
\newcommand{\SonXPlusHat}{S_{\mathrm{cyl}}}
\newcommand{\firstPositiveEigenvalueOfXPlusHat}{\lambda_{\DiracOperatorOnXPlusHat}}

\tikzset{
  XPlus/.pic={
    \draw (0,0.5) -- (4.5,0.5) arc [start angle=150, end angle=-150, radius=1] -- (0,-0.5);
    \node at (5.366,0) {$X_+$};
  },
  XMinus/.pic={
    \draw (0,0.5) -- (-4.5,0.5) arc [start angle=30, end angle=330, radius=1] -- (0,-0.5);
    \node at (-5.366,0) {$X_-$};
  },
  Y/.pic={
    \draw [ultra thick] (0,-0.5) -- (0,0.5);
    \node at (0,0) [right]{$Y$};
  },
  finiteNeckY/.pic={
    \filldraw [very nearly transparent] (-2,0.5) -- (2,0.5) -- (2,-0.5) -- (-2,-0.5);
    \draw [ultra thick] (0,-0.5) -- (0,0.5);
    \draw
      [->] (-2.2,0) -- (2.2,0) node [right]{$u$};
    \draw
      (-2,0.06) -- (-2,-0.06) node [below]{$-4$}
      (2,0.06) -- (2,-0.06) node [below]{$4$};
  },
  infiniteHalfCylinderY/.pic={
    \draw [ultra thick] (0,-0.5) -- (0,0.5);
    \draw (-7,0.5) -- (0,0.5) -- (0,-0.5) -- (-7,-0.5);
    \filldraw [very nearly transparent] (-7,0.5) -- (0,0.5) -- (0,-0.5) -- (-7,-0.5);
    \node at (-5.366,0) {$(-\infty,0) \times Y$};
  },
  verticalX/.pic={
    \draw
      (0,1) arc [start angle=90, end angle=-60, radius=0.25] -- (0.125,0)
      (0,-1) arc [start angle=-90, end angle=60, radius=0.25] -- (0.125,0);
    \draw [dotted]
      (0,1) arc [start angle=90, end angle=240, radius=0.25] -- (-0.125,0)
      (0,-1) arc [start angle=270, end angle=120, radius=0.25] -- (-0.125,0);
  },
  verticalXPlusHat/.pic={
    \draw (0,1) arc [start angle=90, end angle=-60, radius=0.25] -- (0.125,-2);
    \draw [dotted] (0,1) arc [start angle=90, end angle=240, radius=0.25] -- (-0.125,-2);
  },
  infiniteCylinderX/.pic={
    \draw [gray, ->] (-6.5,0) -- (6.5,0) node [right]{$\R$};
    \draw [gray, ->] (0,-1.5) -- (0,1.5) node [above]{$X$};
    \draw
      (-6,1) -- (6,1)
      (-6,-1) -- (6,-1)
      pic at (-4.5,0) {verticalX}
      pic at (4.5,0) {verticalX};
  },
  infiniteCylinderXPlusHat/.pic={
    \draw [->, gray] (-6.5,0) -- (6.5,0) node [right]{$\R$};
    \draw [->, gray] (0,-2) -- (0,1.5) node [above]{$\XPlusHat$};
    \draw
      (-6,1) -- (6,1)
      pic at (-4.5,0) {verticalXPlusHat}
      pic at (4.5,0) {verticalXPlusHat};
  },
}

\begin{document}

\begin{abstract}
  We introduce a \emph{mathematician-friendly} formulation of the \emph{physicist-friendly} derivation of the Atiyah-Patodi-Singer index.
  In a previous paper, motivated by the study of lattice gauge theory, the physicist half of the authors derived a formula expressing the Atiyah-Patodi-Singer index in terms of the eta invariant of \emph{domain-wall fermion Dirac operators} when the base manifold is a flat $4$-dimensional torus.
  In this paper, we generalise this formula to any even dimensional closed Riemannian manifolds, and prove it mathematically rigorously.
  Our proof uses a Witten localisation argument combined with a devised embedding into a cylinder of one dimension higher.
  Our viewpoint sheds some new light on the interplay among the Atiyah-Patodi-Singer boundary condition, domain-wall fermions, and edge modes.
\end{abstract}

\maketitle

%%%%%%%%%%%%%%%%%%%%%%%%%%%%%%%%%%%%%%%%%%%%%%%%%%%%%%%%%%%%%%%%%%%%%%%%%%%%%%%%
\section{Introduction}
%%%%%%%%%%%%%%%%%%%%%%%%%%%%%%%%%%%%%%%%%%%%%%%%%%%%%%%%%%%%%%%%%%%%%%%%%%%%%%%%

The Atiyah-Patodi-Singer index theorem~\cites{MR0397797, MR0397798, MR0397799}, a generalisation of the Atiyah-Singer index theorem to manifolds with boundary, has been attracting attention in condensed matter physics.
For example, Witten used it in~\cite{RevModPhys.88.035001} to describe the bulk-edge correspondence of symmetry-protected topological phases of matter and explain why boundary-localised modes must appear on the boundary of topological insulators.
The boundary correction term of the Atiyah-Patodi-Singer index theorem, the eta invariant, appears as the phase of the edge mode partition function, and the Atiyah-Patodi-Singer index theorem suggests the existence of the bulk topological couplings to restore time-reversal symmetry.
We refer the reader to~\cites{Gromov:2015fda, MR3628684, MR3565832, MR3557925, MR3611419, MR3951702} for related works.

It is, however, somewhat puzzling to relate the Atiyah-Patodi-Singer index and symmetry-protected topological phases of matter.
Considering the Atiyah-Patodi-Singer index, we need a $\Z_2$-grading \emph{chirality} operator; thus, we should consider massless fermions in the bulk and the \emph{non-local} Atiyah-Patodi-Singer spectral boundary condition.
In symmetry-protected topological phases of matter, by contrast, fermions are massive in the bulk and \emph{local} boundary conditions are imposed.
For example, using the Atiyah-Patodi-Singer boundary condition is justified in \cite{Witten-Yonekura}  by rotating the boundary to the temporal direction and regarding it as an intermediate state in the partition function of massive fermion systems.
In a previous paper by the physicist half of the authors~\cite{MR3873281}, we looked at the Atiyah-Patodi-Singer boundary condition in a different light of \emph{domain-wall fermion Dirac operators}.

Domain-wall fermion Dirac operators~\cites{MR779916, MR0468803, Furman-Shamir, Kaplan} are a particular class of massive Dirac operators that have zero-eigenvalue solutions concentrated on small neighbourhoods of separating submanifolds, domain walls.
Using these operators, without imposing any global boundary conditions, the physicist half of the authors gave a physically intuitive reformulation of the Atiyah-Patodi-Singer index in the previous paper~\cite{MR3873281}.
We refer the reader to~\cite{MR3858615} for a different link between the Atiyah-Patodi-Singer index theorem and domain walls.

In this paper, we will pursue our investigation of the relation between the Atiyah-Patodi-Singer index theorem and domain-wall fermion Dirac operators.
We will establish a mathematical formulation of~\cite{MR3873281} based on the embedding trick and a Witten localisation argument~\cites{MR683171, MR2361481} with a new excision theorem of the index under very weak assumptions, which will localise the index to open submanifolds.
See Section~\ref{subsection: Embeddings} and Section~\ref{section: excision} respectively.

Before stating the main theorem, we begin with a formula relating the usual Atiyah-Singer index and the eta invariant.
Let $X$ be a closed oriented Riemannian manifold with $\dim X$ even.
Let $S$ be a $\Z_2$-graded hermitian vector bundle on $X$, and $\gradingOperator_S$ its $\Z_2$-grading operator.
Let $D \colon C^{\infty}(X;S) \to C^{\infty}(X;S)$ be a first-order, formally self-adjoint, elliptic partial differential operator.
We assume that $D$ is an \emph{odd} operator in the sense that it anti-commutes with $\gradingOperator_S$.
Thus, $S$ is decomposed as a direct sum $S = S_+ \oplus S_-$, and we can write
\begin{equation}
  D = \begin{pmatrix}0 & D_- \\ D_+ & 0\end{pmatrix}
\end{equation}
in matrix form.
We define the index of the odd, self-adjoint, elliptic operator $D$ by
\begin{equation}
  \ASindex(D) := \dim \Ker D_+ - \dim \Ker D_- = \tr \big( \restrictedTo{\gradingOperator_S}{\Ker D} \big).
\end{equation}
In physics notation, $\gradingOperator_S = \gamma_5$ and $D = \gamma_5 \slashed{D}$ if $\dim X = 4$, and the index gives the chiral asymmetry of the number of independent left and right zero modes.
Fix $m > 0$, and we consider another self-adjoint, elliptic operator $D + m\gradingOperator_S$.
This is no longer an odd operator.
The eta invariant describes the overall asymmetry of the spectrum of a self-adjoint operator.
Let us recall its definition.
Let $\lambda_j$ run over the eigenvalues of $D + m \gradingOperator_S$.
Note that $\lambda_j \ne 0$ for any $j$.
The eta function of $D + m \gradingOperator_S$ is defined by
\begin{equation}
  \eta(s) := \sum_{\lambda_j} \frac{\sign \lambda_j}{\abs{\lambda_j}^s}
\end{equation}
for $s \in \C$.
This series is absolutely convergent in $\realpart(s) > \dim X$ and admits a meromorphic extension to the whole complex plane.
Atiyah-Patodi-Singer~\cite{MR0397797}*{(3.9)} showed that $\eta(s)$ is holomorphic at $s = 0$.
The special value $\eta(0)$ is called the eta invariant of the operator $D+m\gradingOperator_S$ and denoted by $\eta(D + m \gradingOperator_S)$.
The eta invariant $\eta(D - m \gradingOperator_S)$ is defined similarly.
Now we have a formula
\begin{equation}\label{eq: AS and eta}
  \ASindex (D) = \frac{\eta(D + m\gradingOperator_S) - \eta(D - m \gradingOperator_S)}{2}
\end{equation}
for any $m > 0$.
This formula might be unfamiliar to the reader; however, we can prove it easily, for example, by diagonalising $D^2$ and $\gradingOperator_S$ simultaneously.

The previous paper~\cites{MR3873281} generalised this formula~\eqref{eq: AS and eta} to handle the Atiyah-Patodi-Singer index by considering the \emph{domain-wall fermion Dirac operator}.
Let us first recall the Atiyah-Patodi-Singer index.
Let $Y \subset X$ be a separating submanifold that decomposes $X$ into two compact manifolds $X_+$ and $X_-$ with common boundary $Y$.
We assume that $Y$ has a collar neighbourhood isometric to the standard product $(-4,4) \times Y$ and satisfying $((-4,4) \times Y) \cap X_+ = [0,4) \times Y$.
The coordinate along $(-4,4)$ is denoted by $u$.
We also assume that $S$ and $D$ are standard in the following sense:
there exist a hermitian bundle $E$ on $Y$ and a bundle isometry from $\restrictedTo{S}{(-4,4) \times Y}$ to $\C^2 \otimes E$ such that, under this isometry, $\gradingOperator_S$ corresponds to $\gradingOperator \otimes \id_E$ and $D$ takes the form
\begin{equation}
  D = c \otimes \partial_u + \epsilon \otimes \DiracOperatorOnY =
  \begin{pmatrix}
    0 & \partial_u + \DiracOperatorOnY \\
    -\partial_u + \DiracOperatorOnY & 0
  \end{pmatrix},
\end{equation}
where $\DiracOperatorOnY \colon C^{\infty}(Y;E) \to C^{\infty}(Y;E)$ is a formally self-adjoint, elliptic partial differential operator.
In this paper, we will concentrate on the case when $\DiracOperatorOnY$ has no zero eigenvalues, and assume this condition.
Let $C^{\infty}(X_+;\restrictedTo{S_{\pm}}{X_+}:P_A) := \big\{ f \in C^{\infty}(X_+; \restrictedTo{S_{\pm}}{X_+}) \mid P_A(\restrictedTo{f}{Y}) = 0 \big\}$, where $P_A \colon L^2(Y;E) \to L^2(Y;E)$ denotes the spectral projection onto the span of the eigensections of $A$ with positive eigenvalues.
We define the Atiyah-Patodi-Singer index of $\restrictedTo{D}{X_+}$ by
\begin{equation}
  \APSindex(\restrictedTo{D}{X_+}) := \dim \big( \Ker D_+ \cap C^{\infty}(X_+;\restrictedTo{S_+}{X_+}:P_A) \big) - \dim \big( \Ker D_- \cap C^{\infty}(X_+;\restrictedTo{S_-}{X_+}:P_A) \big).
\end{equation}
To handle the Atiyah-Patodi-Singer index, let us next introduce the \emph{domain-wall fermion Dirac operator}.
Let $\kappa \colon X \to [-1,1]$ be a step function such that $\kappa \equiv \pm 1$ on $X_{\pm} \setminus Y$ and $\kappa \equiv 0$ on $Y$, which is sometimes called a \emph{domain-wall function}.
We call $D + m \kappa \gradingOperator_S$ the \emph{domain-wall fermion Dirac operator}.
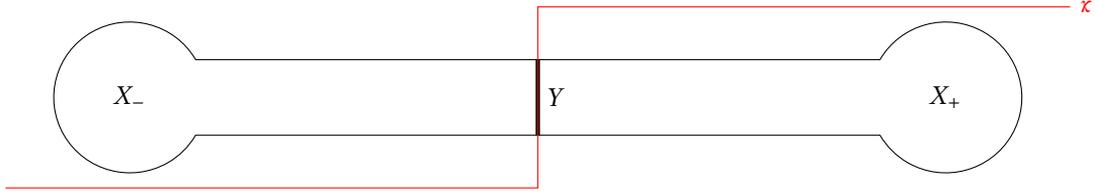
\begin{center}
  \begin{tikzpicture}
    \draw pic {XMinus} pic {Y} pic {XPlus};
    \draw [red]
      (-7,-1.2) -- (0,-1.2) -- (0,1.2) -- (7,1.2)
      node [right] {$\kappa$};
  \end{tikzpicture}
\end{center}
In \cites{MR3873281}, when $X$ is a $4$-dimensional flat manifold, a formula
\begin{equation}\label{eq: APS and eta}
  \APSindex(\restrictedTo{D}{X_+}) = \frac{\eta(D + m\superemph{\kappa}\gradingOperator_S) - \eta(D - m \gradingOperator_S)}{2}
\end{equation}
was derived by first expanding the right hand side with the Fujikawa method~\cite{PhysRevLett.42.1195} and then identifying the result with the left hand side via the Atiyah-Patodi-Singer index theorem.
Our new approach in this paper is mathematically rigorous and reveals the direct link between them; moreover, we will prove the formula~\eqref{eq: APS and eta} for any even-dimensional Riemannian manifolds.
See Theorem~\ref{theorem: main theorem}.

As a warm-up, we will prove the formula~\eqref{eq: AS and eta} in the spirit of our proof of Theorem~\ref{theorem: main theorem}.
The reader can skip this paragraph at first reading.
Consider the cylinder $\R \times X$.
The coordinate along $\R$ is denoted by $s$.
We pull back the bundle $S$ on $X$ to $\R \times X$, which will be denoted by the same symbol.
Let $\kappaAS \colon \R \times X \to [-1,1]$ be a step function such that $\kappaAS \equiv 1$ on $(0,\infty) \times X$ and $\kappaAS \equiv -1$ on $(-\infty,0) \times X$.
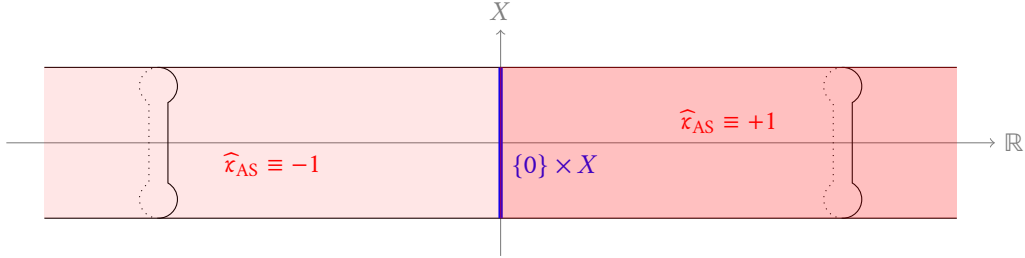
\begin{center}
  \begin{tikzpicture}
    \pic {infiniteCylinderX};
    \draw [blue]
      [ultra thick] (0,-1) -- (0,1)
      node at (0,0) [below right]{$\{0\} \times X$};
    \filldraw [red]
      [very nearly transparent] (-6,1) -- (0,1) -- (0,-1) -- (-6,-1);
    \filldraw [red]
      [nearly transparent] (6,1) -- (0,1) -- (0,-1) -- (6,-1);
    \draw [red]
      node at (-3,0) [below]{$\kappaAS \equiv -1$}
      node at (3,0) [above]{$\kappaAS \equiv +1$};
  \end{tikzpicture}
\end{center}
We introduce a self-adjoint operator $\DiracOperatorOnCylinderAS_m \colon L^2(\R \times X; S \oplus S) \to L^2(\R \times X; S \oplus S)$ defined by
\begin{equation}
  \DiracOperatorOnCylinderAS_m :=
  \begin{pmatrix}
    0 & (D + m\superemph{\kappaAS}\gradingOperator_S) + \partial_s \\
    (D + m\superemph{\kappaAS}\gradingOperator_S) - \partial_s & 0
  \end{pmatrix}.
\end{equation}
We will prove in Proposition~\ref{proposition: product formula} that this is a Fredholm operator and $\ASindex(\DiracOperatorOnCylinderAS_m) = \ASindex(D)$.
We also observe that the constant term in the asymptotic expansion of the heat kernel vanishes on such an odd-dimensional manifold as $\R \times X$.
Thus, by the Atiyah-Patodi-Singer index theorem on cylinders, $\ASindex(\DiracOperatorOnCylinderAS_m)$ can be written only in terms of the eta invariant.
See the discussion around \eqref{eq: product and eta}.
Note that $D + m \kappaAS(\pm 1, \cdot) \gradingOperator_S = D \pm m \gradingOperator_S$.
Hence, we have
\begin{equation}
  \ASindex(D) = \ASindex(\DiracOperatorOnCylinderAS_m) = \frac{\eta(D + m\gradingOperator_S) - \eta(D - m \gradingOperator_S)}{2},
\end{equation}
which proves the formula~\eqref{eq: AS and eta}.

We will generalise the proof above to handle manifolds with boundary.
In the book~\cite{MR2361481}, one of the authors has modified the embedding proof of the Atiyah-Singer index theorem, using a localisation argument of Witten~\cite{MR683171} with \emph{supersymmetric harmonic oscillators}.
In this paper, we will develop another Witten localisation argument with a particular embedding constructed in Section~\ref{subsection: Embeddings} and the \emph{Jackiw-Rebbi solutions} of domain-wall fermionic Dirac operators instead of supersymmetric harmonic oscillators.
We will introduce an operator~\eqref{equation: main operator} that interpolates domain-wall fermion Dirac operators and Atiyah-Patodi-Singer operators.
Our localisation arguments will localise the index of this operator to open submanifolds.
We emphasise here that the ideas behind the proof seem more interesting than the formula~\eqref{eq: APS and eta} itself and useful in other applications.
We also remark that considering spectral flows of the family $\{ D + m \kappaAS(s,\cdot) \gradingOperator_S \}_{s \in [-1,1]}$ seems more appropriate when dealing with various symmetries such as in the ten-fold way of topological insulators~\cite{Kitaev}.

%%%%%%%%%%%%%%%%%%%%%%%%%%%%%%%%%%%%%%%%%%%%%%%%%%%%%%%%%%%%%%%%%%%%%%%%%%%%%%%%
\section{Excision}\label{section: excision}
%%%%%%%%%%%%%%%%%%%%%%%%%%%%%%%%%%%%%%%%%%%%%%%%%%%%%%%%%%%%%%%%%%%%%%%%%%%%%%%%
In this technical section, we will develop, under very weak assumptions, yet another excision formula of the index, which will be a technical basis to the rest of the paper and might be of independent interest.

Before going into details, we first explain some basic ideas underlying our proof of an excision formula, \emph{Witten localisation arguments}~\cite{MR683171}.
Let $U_0 := \{x \in \R^d \mid \abs{x}<2 \}$ and $U_1 := \{x \in \R^d \mid 1 < \abs{x} \}$.
We consider a Dirac-type operator $D$ and a potential term $h$ on $\R^d$.
If $D$ and $h$ anti-commute on $U_1$, then we have $(D+mh)^2 = D^2 + m^2h^2$ on $U_1$ for any $m > 0$.
When $m$ is very large, the second term $m^2h^2$ is also very large.
Hence, eigenmodes with small eigenvalues are very suppressed in this region $U_1$.
In other words, a particle of quantum mechanics is rarely found in the region where the potential energy is very large.
Thus, when $m$ is large enough, the eigenmodes of $(D+mh)^2$ with small eigenvalues are localised and determined on $U_0$.

We set up notation.
Let $Z$ be a complete Riemannian manifold and $S_Z$ a hermitian vector bundle on $Z$.
We denote by $C^{\infty}_c(Z;S_Z)$ the space of compactly supported smooth sections of $S_Z$.
Let $\DiracOperatorOnZ \colon C^{\infty}_c(Z;S_Z) \to C^{\infty}_c(Z;S_Z)$ be a first-order, elliptic partial differential operator that is essentially self-adjoint on $L^2(Z; S_Z)$.
We denote the symbol of $\DiracOperatorOnZ$ by $\sigma_{\DiracOperatorOnZ} \in C^{\infty}(Z; \Hom(T^*Z, \Endo(S_Z)))$.
Let $h$ be a (not necessarily smooth but measurable) self-adjoint endomorphism of $S_Z$ whose eigenvalues are uniformly bounded on $Z$.
For each $m > 0$, we set ${\DiracOperatorOnZ}_m := \DiracOperatorOnZ + mh$.
Throughout this section, $\innerProduct{\cdot}{\cdot}$ will denote a pointwise hermitian inner product, $\abs{\cdot}$ a pointwise norm, and $\norm{\cdot}$ an $L^2$-norm, and denote the exterior derivative of a function $f$ by $df$.

\begin{lemma}\label{lemma: almost splitting}
  Let $\beta_0, \beta_1 \in C^{\infty}(Z;\R)$ satisfy $\beta_0^2 + \beta_1^2 = 1$.
  We have a pointwise equality
  \begin{equation}
    \abs{{\DiracOperatorOnZ}_m(\beta_0 \Phi)}^2 + \abs{{\DiracOperatorOnZ}_m(\beta_1 \Phi)}^2 = \abs{{\DiracOperatorOnZ}_m \Phi}^2 + \abs{\sigma_{\DiracOperatorOnZ}(d\beta_0) \Phi}^2 + \abs{\sigma_{\DiracOperatorOnZ}(d\beta_1) \Phi}^2
  \end{equation}
  for any $\Phi \in C^{\infty}(Z;S_Z)$.
\end{lemma}
\begin{proof}
  Fix $\Phi \in C^{\infty}(Z;S_Z)$.
  Since $\DiracOperatorOnZ(\beta_0 \Phi) = \beta_0 (\DiracOperatorOnZ \Phi) + \sigma_{\DiracOperatorOnZ}(d\beta_0) \Phi$, we have ${\DiracOperatorOnZ}_m(\beta_0 \Phi) = \beta_0 ({\DiracOperatorOnZ}_m \Phi) + \sigma_{\DiracOperatorOnZ}(d\beta_0) \Phi$.
  Hence, $\abs{{\DiracOperatorOnZ}_m(\beta_0 \Phi)}^2 = \beta_0^2 \abs{{\DiracOperatorOnZ}_m \Phi}^2 + \abs{\sigma_{\DiracOperatorOnZ}(d\beta_0) \Phi}^2 + 2 \realpart\innerProduct{\beta_0({\DiracOperatorOnZ}_m \Phi)}{\sigma_{\DiracOperatorOnZ}(d\beta_0)\Phi}$.
  Thus, we have
  \begin{multline}
    \abs{{\DiracOperatorOnZ}_m(\beta_0 \Phi)}^2 + \abs{{\DiracOperatorOnZ}_m(\beta_1 \Phi)}^2 = \abs{{\DiracOperatorOnZ}_m \Phi}^2 + \abs{\sigma_{\DiracOperatorOnZ}(d\beta_0) \Phi}^2  + \abs{\sigma_{\DiracOperatorOnZ}(d\beta_1) \Phi}^2 \\
    + 2 \realpart\innerProduct{\beta_0({\DiracOperatorOnZ}_m \Phi)}{\sigma_{\DiracOperatorOnZ}(d\beta_0)\Phi} + 2 \realpart\innerProduct{\beta_1({\DiracOperatorOnZ}_m \Phi)}{\sigma_{\DiracOperatorOnZ}(d\beta_1)\Phi}.
  \end{multline}
  The assumption $\beta_0^2 + \beta_1^2 = 1$ implies $2 \innerProduct{\beta_0({\DiracOperatorOnZ}_m \Phi)}{\sigma_{\DiracOperatorOnZ}(d\beta_0)\Phi} + 2 \innerProduct{\beta_1({\DiracOperatorOnZ}_m \Phi)}{\sigma_{\DiracOperatorOnZ}(d\beta_1)\Phi} = 0$.
\end{proof}

\begin{lemma}\label{lemma: integral equality}
  Under the assumption of Lemma~\ref{lemma: almost splitting}, suppose further that $h$ is smooth and anti-commutes with $\DiracOperatorOnZ$ on $\supp \beta_1$.
  Then, we have an $L^2$-integral equality
  \begin{equation}
    m^2 \norm{h \beta_1 \Phi}^2 \le \norm{{\DiracOperatorOnZ}_m \Phi}^2 + \norm{\sigma_{\DiracOperatorOnZ}(d\beta_0) \Phi}^2 + \norm{\sigma_{\DiracOperatorOnZ}(d\beta_1) \Phi}^2
  \end{equation}
  for any $\Phi \in C^{\infty}_c(Z;S_Z)$.
\end{lemma}
\begin{proof}
  Fix $\Phi \in C^{\infty}_c(Z;S_Z)$.
  By Lemma~\ref{lemma: almost splitting}, we have an inequality
  \begin{equation}
    \int_Z \abs{{\DiracOperatorOnZ}_m(\beta_1 \Phi)}^2 \,d\mu \le \int_Z \big( \abs{{\DiracOperatorOnZ}_m \Phi}^2 + \abs{\sigma_{\DiracOperatorOnZ}(d\beta_0) \Phi}^2  + \abs{\sigma_{\DiracOperatorOnZ}(d\beta_1) \Phi}^2 \big) \,d\mu.
  \end{equation}
  Since $\DiracOperatorOnZ$ and $h$ anti-commute on $\supp \beta_1$, we deduce that
  \begin{equation}\begin{split}
    \int_Z \innerProduct{\DiracOperatorOnZ(\beta_1 \Phi)}{h \beta_1 \Phi} \,d\mu =& \int_Z \innerProduct{\beta_1 \Phi}{\DiracOperatorOnZ(h \beta_1 \Phi)} \,d\mu \\
    =& \int_Z \innerProduct{\beta_1 \Phi}{-h \DiracOperatorOnZ(\beta_1 \Phi)} \,d\mu = - \int_Z \innerProduct{h\beta_1 \Phi}{\DiracOperatorOnZ(\beta_1 \Phi)} \,d\mu.
  \end{split}\end{equation}
  Hence, $\realpart\bracket{\DiracOperatorOnZ(\beta_1 \Phi)}{h\beta_1\Phi}_{L^2} = 0$, and we have $\norm{{\DiracOperatorOnZ}_m(\beta_1 \Phi)}^2 = \norm{\DiracOperatorOnZ(\beta_1 \Phi)}^2 + \norm{m h \beta_1 \Phi}^2$.
  Thus,
  \begin{equation}
    \norm{m h \beta_1 \Phi}^2 \le \norm{{\DiracOperatorOnZ}_m(\beta_1 \Phi)}^2 \le \norm{{\DiracOperatorOnZ}_m \Phi}^2 + \norm{\sigma_{\DiracOperatorOnZ}(d\beta_0) \Phi}^2 + \norm{\sigma_{\DiracOperatorOnZ}(d\beta_1) \Phi}^2,
  \end{equation}
  as required.
\end{proof}

\begin{lemma}\label{lemma: weak localisation}
  Let $Z = U_0 \cup U_1$ be an open covering of $Z$.
  Let $1 = \gamma_0^2 + \gamma_1^2$ be a smooth partition of unity subordinate to $U_0$ and $U_1$.
  We assume the following three conditions:
  \begin{enumerate}
    \item $h$ is smooth and anti-commutes with $\DiracOperatorOnZ$ on $U_1$.\label{assumption: wl-1}
    \item the eigenvalues of $h^2$ are greater than or equal  to $1$ on $U_1$.\label{assumption: wl-2}
    \item the eigenvalues of $\sigma_{\DiracOperatorOnZ}(d\gamma_0)$ and $\sigma_{\DiracOperatorOnZ}(d\gamma_1)$ are bounded on $U_0 \cap U_1$.\label{assumption: wl-3}
  \end{enumerate}
  Then, for any $\Lambda \ge 0$ and $\Phi \in C^{\infty}_c(Z;S_Z)$ with $\norm{{\DiracOperatorOnZ}_m \Phi}^2 \le \Lambda^2 \norm{\Phi}^2$, we have
  \begin{equation}
    m^2\norm{\gamma_1 \Phi}^2 \le (C_1^2 + \Lambda^2) \norm{\Phi}^2,
  \end{equation}
  where we set
  \begin{equation}
    C_1^2
    := \sup_{x \in U_0 \cap U_1} \big( \abs{\sigma_{\DiracOperatorOnZ}(d\gamma_0)}^2 + \abs{\sigma_{\DiracOperatorOnZ}(d\gamma_1)}^2 \big)
    = \sup_{x \in U_0 \cap U_1} \sup_{\phi \in S_x} \left( \frac{\abs{\sigma_{\DiracOperatorOnZ}(d\gamma_0)\phi}^2}{\abs{\phi}^2} + \frac{\abs{\sigma_{\DiracOperatorOnZ}(d\gamma_1)\phi}^2}{\abs{\phi}^2} \right).
  \end{equation}
\end{lemma}
\begin{proof}
  Fix $\Lambda \ge 0$ and $\Phi \in C^{\infty}_c(Z;S_Z)$ with $\norm{{\DiracOperatorOnZ}_m \Phi}^2 \le \Lambda^2 \norm{\Phi}^2$.
  By assumption~\eqref{assumption: wl-2}, we have $m^2 \norm{h \gamma_1 \Phi}^2 \ge m^2 \norm{\gamma_1 \Phi}^2$.
  By definition of $C_1$, we have $\norm{\sigma_{\DiracOperatorOnZ}(d\gamma_0) \Phi}^2 + \norm{\sigma_{\DiracOperatorOnZ}(d\gamma_1) \Phi}^2 \le C_1^2 \norm{\Phi}^2$.
  By assumption~\eqref{assumption: wl-1}, we can use Lemma~\ref{lemma: integral equality}.
  Thus, we obtain
  \begin{equation}\begin{split}
    m^2\norm{\gamma_1 \Phi}^2 &\le m^2 \norm{h \gamma_1 \Phi}^2 \\
    &\le \norm{{\DiracOperatorOnZ}_m \Phi}^2 + \norm{\sigma_{\DiracOperatorOnZ}(d\gamma_0) \Phi}^2 + \norm{\sigma_{\DiracOperatorOnZ}(d\gamma_1) \Phi}^2 \\
    &\le \norm{{\DiracOperatorOnZ}_m \Phi}^2 + C_1^2 \norm{\Phi}^2 \le \Lambda^2 \norm{\Phi}^2 + C_1^2 \norm{\Phi}^2 = (C_1^2 + \Lambda^2) \norm{\Phi}^2,
  \end{split}\end{equation}
  as required.
\end{proof}

\begin{lemma}\label{lemma: good partitions of unity}
  Let $Z = U_0 \cup U_1$ be an open covering of $Z$.
  Let $1 = \eta_0 + (1-\eta_0)$ be a smooth partition of unity subordinate to $U_0$ and $U_1$.
  We assume that $\abs{\sigma_{\DiracOperatorOnZ}(d\eta_0)}$ is bounded on $U_0 \cap U_1$.
  Then, there exist smooth partitions of unity $1 = \beta_0^2 + \beta_1^2 = \gamma_0^2 + \gamma_1^2$ subordinate to $U_0$ and $U_1$ such that both $\abs{\sigma_{\DiracOperatorOnZ}(d\beta_0)}^2 + \abs{\sigma_{\DiracOperatorOnZ}(d\beta_1)}^2$ and $\abs{\sigma_{\DiracOperatorOnZ}(d\gamma_0)}^2 + \abs{\sigma_{\DiracOperatorOnZ}(d\gamma_1)}^2$ are bounded on $U_0 \cap U_1$, and that $\gamma_1 \equiv 1$ on $(\supp d\beta_0) = (\supp d\beta_1)$.
\end{lemma}
\begin{proof}
  Let $\beta \colon [0,1] \to [0,1]$ and $\gamma \colon [0,1] \to [0,1]$ be smooth cut-off functions such that $\beta^2 + (1-\beta^2) = 1$ and $\gamma^2 + (1-\gamma^2) = 1$, and that $\gamma \equiv 0$ on $[0,1/4]$, $\beta \equiv 0$ on $[0,1/2]$, $\gamma \equiv 1$ on $[1/2,1]$, and $\beta \equiv 1$ on $[3/4,1]$.
  
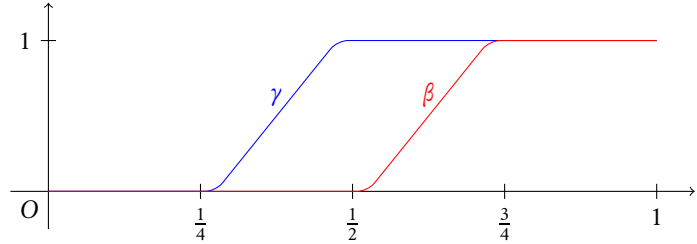
\begin{center}
    \begin{tikzpicture}
      \draw [->] (-0.5,0) -- (8.5,0);
      \draw [->] (0,-0.5) -- (0,2.5);
      \node at (0,0) [below left]{$O$};
      \draw
        (2,0.1) -- (2,-0.1) node [below]{$\frac{1}{4}$}
        (4,0.1) -- (4,-0.1) node [below]{$\frac{1}{2}$}
        (6,0.1) -- (6,-0.1) node [below]{$\frac{3}{4}$}
        (8,0.1) -- (8,-0.1) node [below]{$1$};
      \draw (0.1,2) -- (-0.1,2) node [left]{$1$};
      \draw [blue]
        [rounded corners] (0,0) -- (2.2,0) -- (3.8,2) -- (8,2)
        node at (3,1) [above]{$\gamma$};
      \draw [red]
        [rounded corners] (0,0) -- (4.2,0) -- (5.8,2) -- (8,2)
        node at (5,1) [above]{$\beta$};
    \end{tikzpicture}
  \end{center}
  We set $\beta_1 := \beta \circ (1-\eta_0)$ and $\gamma_1 := \gamma \circ (1-\eta_0)$, which clearly satisfy the claimed properties.
\end{proof}

\begin{proposition}\label{proposition: localisation}
  Let $Z = U_0 \cup U_1$ and $1 = \eta_0 + (1-\eta_0)$ satisfy the assumptions of Lemma~\ref{lemma: good partitions of unity}.
  Let $1 = \beta_0^2 + \beta_1^2 = \gamma_0^2 + \gamma_1^2$ be partitions of unity constructed in Lemma~\ref{lemma: good partitions of unity}.
  We also assume that $h$ satisfies the conditions~\eqref{assumption: wl-1} and~\eqref{assumption: wl-2} of Lemma~\ref{lemma: weak localisation}.
  Then, there exists a constant $C_0 > 0$ that depends only on $\eta_0$ and $\sigma_{\DiracOperatorOnZ}$ such that, for any $\Lambda \ge 0$ and $\Phi \in C^{\infty}_c(Z;S_Z)$ with $\norm{{\DiracOperatorOnZ}_m \Phi}^2 \le \Lambda^2 \norm{\Phi}^2$, we have
  \begin{gather}
    \norm{{\DiracOperatorOnZ}_m (\beta_0 \Phi)}^2 \le \left( \Lambda^2 + C_0^2\frac{C_0^2 + \Lambda^2}{m^2} \right) \norm{\Phi}^2\label{eq: eigenvalue comparison} \\
    \shortintertext{and}
    \left( 1 - \frac{C_0^2 + \Lambda^2}{m^2} \right) \norm{\Phi}^2 \le \norm{\beta_0 \Phi}^2.\label{eq: eigenfunction localisation}
  \end{gather}
\end{proposition}
\begin{proof}
  Fix $\Lambda \ge 0$ and $\Phi \in C^{\infty}_c(Z;S_Z)$ with $\norm{{\DiracOperatorOnZ}_m \Phi}^2 \le \Lambda^2 \norm{\Phi}^2$.
  Set $C_1^2 := \sup_{U_0 \cap U_1} \big( \abs{\sigma_{\DiracOperatorOnZ}(d\gamma_0)}^2 + \abs{\sigma_{\DiracOperatorOnZ}(d\gamma_1)}^2 \big)$ and $C_2^2 := \sup_{U_0 \cap U_1} \big( \abs{\sigma_{\DiracOperatorOnZ}(d\beta_0)}^2 + \abs{\sigma_{\DiracOperatorOnZ}(d\beta_1)}^2 \big)$.

  We first show~\eqref{eq: eigenvalue comparison}.
  By Lemma~\ref{lemma: weak localisation}, we have
  \begin{equation}
    m^2 \norm{\gamma_1 \Phi}^2 \le (C_1^2 + \Lambda^2) \norm{\Phi}^2.
  \end{equation}
  Since $\gamma_1 \equiv 1$ on $(\supp d\beta_0) = (\supp d\beta_1)$, we have
  \begin{equation}
    \norm{\sigma_{\DiracOperatorOnZ}(d\beta_0) \Phi}^2 + \norm{\sigma_{\DiracOperatorOnZ}(d\beta_1) \Phi}^2 \le C_2^2 \norm{\gamma_1 \Phi}^2.
  \end{equation}
  Thus, we obtain
  \begin{equation}
    \norm{\sigma_{\DiracOperatorOnZ}(d\beta_0) \Phi}^2 + \norm{\sigma_{\DiracOperatorOnZ}(d\beta_1) \Phi}^2 \le C_2^2 \frac{C_1^2 + \Lambda^2}{m^2} \norm{\Phi}^2.
  \end{equation}
  By Lemma~\ref{lemma: almost splitting}, we have
  \begin{equation}
    \norm{{\DiracOperatorOnZ}_m (\beta_0 \Phi)}^2 \le \norm{{\DiracOperatorOnZ}_m \Phi}^2 + \norm{\sigma_{\DiracOperatorOnZ}(d\beta_0) \Phi}^2 + \norm{\sigma_{\DiracOperatorOnZ}(d\beta_1) \Phi}^2.
  \end{equation}
  Consequently, we have
  \begin{equation}
    \norm{{\DiracOperatorOnZ}_m (\beta_0 \Phi)}^2 \le \Lambda^2 \norm{\Phi}^2 + C_2^2 \frac{C_1^2 + \Lambda^2}{m^2} \norm{\Phi}^2.
  \end{equation}
  Now set $C_0 := \max \{C_1, C_2\}$, which yields~\eqref{eq: eigenvalue comparison}.

  Next we prove~\eqref{eq: eigenfunction localisation}.
  Since $\beta_0^2 + \beta_1^2 = 1$, we have $\norm{\beta_0 \Phi}^2 + \norm{\beta_1 \Phi}^2 = \norm{\Phi}^2$.
  By Lemma~\ref{lemma: weak localisation}, we have $m^2\norm{\beta_1 \Phi}^2 \le (C_2^2 + \Lambda^2) \norm{\Phi}^2 \le (C_0^2 + \Lambda^2) \norm{\Phi}^2$.
  This completes the proof.
\end{proof}

\begin{proposition}\label{proposition: excision}
  Let $(Z = U_0 \cup U_1, 1 = \eta_0 + (1-\eta_0), S_Z, \DiracOperatorOnZ, h)$ and $(Z' = U'_0 \cup U'_1, 1 = \eta'_0 + (1-\eta'_0), S'_{Z'}, \DiracOperatorOnZ', h')$ be two sets of data as above that satisfy the assumptions of Proposition~\ref{proposition: localisation}.
  We assume that $\DiracOperatorOnZ$ coincides with $\DiracOperatorOnZ'$ on $U_0 \cong U'_0$ in the sense that there exists an isometry $\tau \colon U_0 \to U'_0$ covered by a bundle isometry $\tilde{\tau} \colon \restrictedTo{S_Z}{U_0} \to \restrictedTo{S'_{Z'}}{U'_0}$ such that $\tilde{\tau}^{-1} \circ L'_m \circ \tilde{\tau} = L_m$.
  Then, there exists a constant $C > 0$ that depends only on $\eta_0$, $\eta'_0$, $\sigma_{\DiracOperatorOnZ}$, and $\sigma_{\DiracOperatorOnZ'}$ such that the following holds.
  Fix $\Lambda_2 > \Lambda_1 > \Lambda_0 \ge 0$ and $m >0$.
  If ${\DiracOperatorOnZ}_m$ has only discrete spectrum\footnote{The discrete spectrum of a self-adjoint operator consists of isolated eigenvalues with finite multiplicity.} in $[-\Lambda_0, \Lambda_0]$ and has spectral gaps
  \begin{equation}\label{assumption: spectral gap}
    \big( \spectrum {\DiracOperatorOnZ}_m \big) \cap \big( \big[-\Lambda_2, -\Lambda_0 \big) \cup \big( \Lambda_0, \Lambda_2 \big] \big) = \emptyset
  \end{equation}
  and if
  \begin{equation}\label{assumption: heavy mass}
    m^2 > \max \left\{ \frac{(C^2+\Lambda_2^2)(C^2+\Lambda_1^2)}{\Lambda_2^2-\Lambda_1^2}, \frac{(C^2+\Lambda_1^2)(C^2+\Lambda_0^2)}{\Lambda_1^2-\Lambda_0^2}, (C^2+\Lambda_2^2) \right\},
  \end{equation}
  then ${\DiracOperatorOnZ'}_m$ also has only discrete spectrum in $[-\Lambda_1, \Lambda_1]$ and the number of eigenvalues of ${\DiracOperatorOnZ'}_m$ in $[-\Lambda_1, \Lambda_1]$ counted with multiplicity is equal to that of $\DiracOperatorOnZ_m$ in $[-\Lambda_0, \Lambda_0]$.
\end{proposition}
\begin{remark}
  The spectral gap condition~\eqref{assumption: spectral gap} is only imposed on $\DiracOperatorOnZ_m$.
\end{remark}
\begin{proof}
  Let $E_0$ be the span of the $L^2$-eigensections of ${\DiracOperatorOnZ}_m$ with eigenvalues in $[-\Lambda_0, \Lambda_0]$ and $E_2$ with eigenvalues in $[-\Lambda_2, \Lambda_2]$.
  Let $\Pi_0 \colon L^2(Z;S_Z) \to E_0$ and $\Pi_2 \colon L^2(Z;S_Z) \to E_2$ be the $L^2$-orthogonal projections.
  The spectral gap assumption~\eqref{assumption: spectral gap} on $\DiracOperatorOnZ_m$ implies $E_0$ is a finite dimensional vector space and $E_0 = E_2$.
  Let $\Pi'_1 \colon L^2(Z';S'_{Z'}) \to L^2(Z';S'_{Z'})$ be the spectral projection for $\DiracOperatorOnZ'_m$ associated with $[-\Lambda_1, \Lambda_1]$.
  Let $E'_1 := \Image \Pi'_1$.
  We will show that $E_0 \cong E'_1$.
  
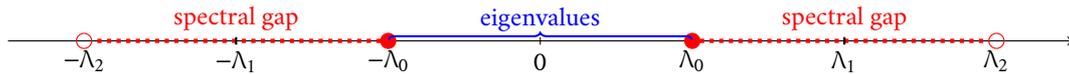
\begin{center}
    \begin{tikzpicture}
      \draw [->] (-7,0) -- (7,0);
      \draw (0,0.05) -- (0,-0.05) node [below]{$0$};

      \node at (2,0) [below]{$\Lambda_0$};
      \filldraw [red] (2,0) circle [radius=0.1];
      \draw [red, ultra thick, dotted] (2.1,0) -- node [above]{spectral gap} (5.9,0);
      \draw [red] (6,0) circle [radius=0.1];
      \node at (6,0) [below]{$\Lambda_2$};

      \node at (-2,0) [below]{$-\Lambda_0$};
      \filldraw [red] (-2,0) circle [radius=0.1];
      \draw [red, ultra thick, dotted] (-2.1,0) -- node [above]{spectral gap} (-5.9,0);
      \draw [red] (-6,0) circle [radius=0.1];
      \node at (-6,0) [below]{$-\Lambda_2$};

      \draw [decorate, decoration={brace, raise=0.5}, blue, thick] (-2,0) -- node [above] {eigenvalues} (2,0);

      \draw [thick] (4,0.05) -- (4,-0.05) node at (4,0) [below]{$\Lambda_1$};
      \draw [thick] (-4,0.05) -- (-4,-0.05) node at (-4,0) [below]{$-\Lambda_1$};
    \end{tikzpicture}
  \end{center}

  Let $1 = \beta_0^2 + \beta_1^2$ and $1 = (\beta'_0)^2 + (\beta'_1)^2$ be partitions of unity in Lemma~\ref{lemma: good partitions of unity}.
  We regard, via $\tau$ and $\tilde{\tau}$, each section of $S_Z$ supported in $U_0$ as a section of $S'_{Z'}$ supported in $U'_0$.
  We define a linear map $\rho \colon E_0 \to E'_1$ by
  \begin{equation}
    \Phi \mapsto \Pi'(\beta_0 \Phi)
  \end{equation}
  and a linear map $\rho' \colon E'_1 \to E_2$ similarly.
  We will prove that $\rho$ and $\rho'$ are isomorphisms.

  We first show that $\rho$ is injective.
  Fix $\Phi \in \Ker (\rho)$.
  Assume that $\Phi \ne 0$.
  The assumption~\eqref{assumption: heavy mass} and the inequality~\eqref{eq: eigenfunction localisation} implies that $\beta_0 \Phi \ne 0$.
  We define
  \begin{equation}
    C := \max \big\{ C_0(\eta_0, \sigma_{\DiracOperatorOnZ}), C_0(\eta'_0, \sigma_{\DiracOperatorOnZ'}) \big\},
  \end{equation}
  where $C_0$ is the constant in Proposition~\ref{proposition: localisation}.
  Then, by Proposition~\ref{proposition: localisation}, we have
  \begin{equation}
    \begin{aligned}
      \norm{\DiracOperatorOnZ_m(\beta_0 \Phi)}^2 &\le \left( \Lambda_0^2 + C^2\frac{C^2 +\Lambda_0^2}{m^2} \right) \norm{\Phi}^2 \\
      &\le \left( \Lambda_0^2 + C^2\frac{C^2 + \Lambda_0^2}{m^2} \right) \left( 1 - \frac{C^2 + \Lambda_0^2}{m^2} \right)^{-1} \norm{\beta_0 \Phi}^2.
    \end{aligned}
  \end{equation}
  The assumption~\eqref{assumption: heavy mass} implies
  \begin{equation}
    \left( \Lambda_0^2 + C^2\frac{C^2 + \Lambda_0^2}{m^2} \right) \left( 1 - \frac{C^2 + \Lambda_0^2}{m^2} \right)^{-1} < \Lambda_1^2.
  \end{equation}
  Hence, we have $\norm{\DiracOperatorOnZ_m(\beta_0 \Phi)}^2 < \Lambda_1^2 \norm{\beta_0\Phi}^2$.
  Since $\beta_0 \Phi$ is supported in $U_0$ and $\DiracOperatorOnZ_m$ coincides with ${\DiracOperatorOnZ'}_m$ on $U_0$, we have $\norm{\DiracOperatorOnZ_m(\beta_0 \Phi)} = \norm{{\DiracOperatorOnZ'}_m(\beta_0 \Phi)}$.
  Thus, we have
  \begin{equation}
    \norm{\DiracOperatorOnZ'_m(\beta_0 \Phi)}^2 < \Lambda_1^2 \norm{\beta_0\Phi}^2,
  \end{equation}
  which implies $\beta_0\Phi = 0$ by the definition of $\Pi'$.
  This contradicts the assumption.
  Thus, $\rho$ is injective.

  We next show that $\rho'$ is injective.
  Fix $\Phi' \in \Ker(\rho')$.
  In the same way as above, we have
  \begin{equation}
    \begin{aligned}
      \norm{\DiracOperatorOnZ_m (\beta'_0 \Phi')}^2 = \norm{\DiracOperatorOnZ'_m(\beta'_0 \Phi')}^2 &\le \left( \Lambda_1^2 + C^2\frac{C^2+\Lambda_1^2}{m^2} \right) \norm{\Phi'}^2 \\
      &\le \left( \Lambda_1^2 + C^2\frac{C^2 +\Lambda_1^2}{m^2} \right) \left( 1 - \frac{C^2 +\Lambda_1^2}{m^2} \right)^{-1} \norm{\beta'_0 \Phi'}^2 \\
      &< \Lambda_2^2 \norm{\beta'_0 \Phi'}^2,
    \end{aligned}
  \end{equation}
  which implies $\rho'$ is also injective.

  We have now shown that $\rho \colon E_0 \to E'_1$ and $\rho' \colon E'_1 \to E_2$ are injective.
  Since $E_0$ is finite dimensional and $E_0 =E_2$, it follows that $E_0 \cong E'_1$.
  The proof is complete.
\end{proof}

\begin{theorem}\label{theorem: localisation}
  Let $(Z = U_0 \cup U_1, 1 = \eta_0 + (1-\eta_0), S_Z, \DiracOperatorOnZ, h)$ and $(Z' = U'_0 \cup U'_1, 1 = \eta'_0 + (1-\eta'_0), S'_{Z'}, \DiracOperatorOnZ', h')$ be two sets of data as above.
  We make the following assumptions:
  \begin{enumerate}
    \item $\abs{\sigma_{\DiracOperatorOnZ}(d\eta_0)}$ is bounded on $U_0 \cap U_1$, and $\abs{\sigma_{\DiracOperatorOnZ'}(d\eta_0')}$ is bounded on $U_0' \cap U_1'$.
    \item $h$ is smooth and anti-commutes with $\DiracOperatorOnZ$ on $U_1$, and $h'$ is smooth and anti-commutes with $\DiracOperatorOnZ'$ on $U_1'$.
    \item the eigenvalues of $h^2$ are greater than or equal to $1$ on $U_1$, and the eigenvalues of $(h')^2$ are greater than or equal to $1$ on $U_1'$.
    \item there exists an isometry $\tau \colon U_0 \to U'_0$ covered by a bundle isometry $\tilde{\tau} \colon \restrictedTo{S_Z}{U_0} \to \restrictedTo{S'_{Z'}}{U'_0}$ such that $\tilde{\tau}^{-1} \circ L'_m \circ \tilde{\tau} = L_m$.
    \item $S_Z$ and $S'_{Z'}$ are $\Z_2$-graded and that $\DiracOperatorOnZ$, $\DiracOperatorOnZ'$, $h$, and $h'$ are odd operators.
  \end{enumerate}
  Then, there exists a constant $C > 0$ that depends only on $\eta_0$, $\eta'_0$, $\sigma_{\DiracOperatorOnZ}$, and $\sigma_{\DiracOperatorOnZ'}$ such that the following holds.
  Fix $\Lambda > 0$ and $m >0$.
  If $\DiracOperatorOnZ_m$ is a Fredholm operator with $(\spectrum \DiracOperatorOnZ_m) \cap [-\Lambda, \Lambda] = \{0\}$ and $m > 2(C^2 + \Lambda^2)/\Lambda$, then $\DiracOperatorOnZ'_m$ is also a Fredholm operator and we have
  \begin{equation}
    \ASindex (\DiracOperatorOnZ_m) = \ASindex (\DiracOperatorOnZ'_m).
  \end{equation}
\end{theorem}
\begin{proof}
  Fix $\Lambda > 0$ and $m >0$ with $(\spectrum \DiracOperatorOnZ_m) \cap [-\Lambda, \Lambda] = \{0\}$ and $m > 2(C^2 + \Lambda^2)/\Lambda$.
  Set $\Lambda_2 := \Lambda$, $\Lambda_1 := \Lambda/\sqrt{2}$, and $\Lambda_0 := 0$.
  Then, $(\Lambda_2, \Lambda_1, \Lambda_0)$ satisfies \eqref{assumption: spectral gap} and \eqref{assumption: heavy mass} of Proposition~\ref{proposition: excision}.
  Note that spectral projections commute with grading operators.
  Thus, we conclude from Proposition~\ref{proposition: excision} that $\ASindex (\DiracOperatorOnZ_m) = \ASindex (\DiracOperatorOnZ'_m)$.
\end{proof}

%%%%%%%%%%%%%%%%%%%%%%%%%%%%%%%%%%%%%%%%%%%%%%%%%%%%%%%%%%%%%%%%%%%%%%%%%%%%%%%%
\section{The main theorem}
%%%%%%%%%%%%%%%%%%%%%%%%%%%%%%%%%%%%%%%%%%%%%%%%%%%%%%%%%%%%%%%%%%%%%%%%%%%%%%%%

\subsection{Notation}

We define $c$, $\epsilon$, and $\gradingOperator$ by
\begin{equation}\label{eq: clifford generators}
  c = \begin{pmatrix}0 & 1 \\ -1 & 0\end{pmatrix}, \quad
  \epsilon = \begin{pmatrix}0 & 1 \\ 1 & 0\end{pmatrix}, \quad \text{ and }
  \gradingOperator = \begin{pmatrix}1 & 0 \\ 0 & -1\end{pmatrix}.
\end{equation}
Then, $c^2 = -1$, $\epsilon^2 = \gradingOperator^2 = 1$, $\gradingOperator = c\epsilon$, and they anti-commute.

Let $X$ be a closed oriented Riemannian manifold with $\dim X$ even.
Let $S$ be a $\Z_2$-graded hermitian vector bundle on $X$, and $\gradingOperator_S$ its $\Z_2$-grading operator.
Let $D \colon C^{\infty}(X;S) \to C^{\infty}(X;S)$ be a first-order, formally self-adjoint, elliptic partial differential operator that anti-commutes with $\gradingOperator_S$.
Let $Y \subset X$ be a separating submanifold that decomposes $X$ into two compact manifolds $X_+$ and $X_-$ with common boundary $Y$.
Let $\kappa \colon X \to [-1,1]$ be an $L^{\infty}$-function such that $\kappa \equiv \pm 1$ on $X_{\pm} \setminus Y$.

We assume that $Y$ has a collar neighbourhood isometric to the standard product $(-4,4) \times Y$ and satisfying $((-4,4) \times Y) \cap X_+ = [0,4) \times Y$.
The coordinate along $(-4,4)$ is denoted by $u$.
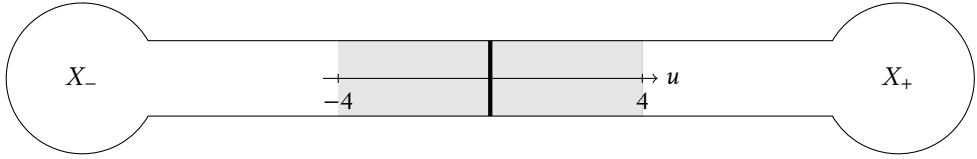
\begin{center}
  \begin{tikzpicture}
    \draw pic {XMinus} pic {finiteNeckY} pic {XPlus};
  \end{tikzpicture}
\end{center}
We also assume that $S$ and $D$ are standard in the following sense:
there exist a hermitian bundle $E$ on $Y$ and a bundle isometry from $\restrictedTo{S}{(-4,4) \times Y}$ to $\C^2 \otimes E$ such that, under this isometry, $\gradingOperator_S$ corresponds to $\gradingOperator \otimes \id_E$ and $D$ takes the form
\begin{equation}
  D = c \otimes \partial_u + \epsilon \otimes \DiracOperatorOnY =
  \begin{pmatrix}
    0 & \partial_u + \DiracOperatorOnY \\
    -\partial_u + \DiracOperatorOnY & 0
  \end{pmatrix},
\end{equation}
where $\DiracOperatorOnY \colon C^{\infty}(Y;E) \to C^{\infty}(Y;E)$ is a formally self-adjoint, elliptic partial differential operator.
In this paper, we will concentrate on the case when $\DiracOperatorOnY$ has no zero eigenvalues, and assume this condition.

\subsection{Spectral gaps}

As a first step, we will consider the spectral gap of domain-wall fermion Dirac operators.
We begin with the one-dimensional operator
\begin{equation}
  c \partial_t + m\signfunction_0\epsilon =
  \begin{pmatrix}
    0 & \partial_t + m \signfunction_0 \\
    -\partial_t + m \signfunction_0 & 0
  \end{pmatrix}
  \colon C^{\infty}_c(\R;\C^2) \to L^2(\R;\C^2),
\end{equation}
where $\signfunction_0 \colon \R \to \R$ is a sign function such that $\signfunction_0(\pm t) = \pm 1$ for $t > 0$.
It is well-known that the operator $(c \partial_t + m \signfunction_0\epsilon)$ is essentially self-adjoint on $L^2(\R;\C^2)$ and that it has essential spectrum equal to $(-\infty, -m] \cup [m,\infty)$ and $0$ is a unique and simple eigenvalue.
See, for example,~\cite{MR3633291}*{Theorem 4.2}.
We note that
\begin{equation}
  \frac{d}{dt} e^{-m\abs{t}} = -m \signfunction_0 e^{-m\abs{t}},
\end{equation}
for any $m > 0$.
Let $v_- = \transpose{(0,1)}$, which satisfies $\gradingOperator v_- = - v_-$.
Then, $e^{-m\abs{t}}v_-$ satisfies
\begin{equation}
  \big( c \partial_t + m\signfunction_0\epsilon \big) \big( e^{-m\abs{t}}v_- \big) = 0,
\end{equation}
which is called the Jackiw-Rebbi solution~\cite{MR0468803}.
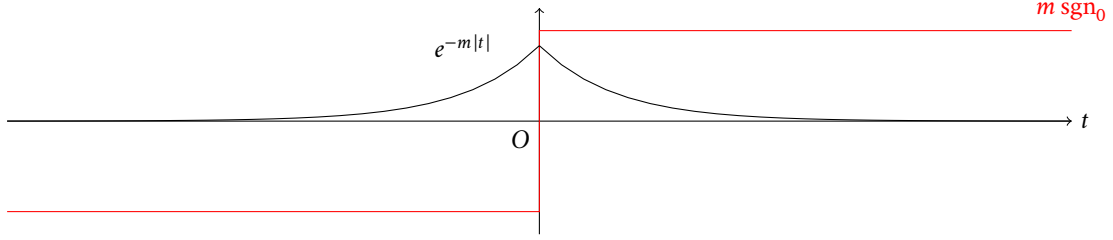
\begin{center}
  \begin{tikzpicture}
    \draw [->] (-7,0) -- (7,0) node [right]{$t$};
    \draw [->] (0,-1.5) -- (0,1.5);
    \node at (0,0) [below left]{$O$};
    \draw [red]
      (-7,-1.2) -- (0,-1.2) -- (0,1.2) -- (7,1.2)
      node [above]{$m\signfunction_0$};
    \draw
      [domain=0:7] plot (\x,{exp(-\x)})
      [domain=-7:0] plot (\x,{exp(\x)})
      node at (-1,1) {$e^{-m\abs{t}}$};
  \end{tikzpicture}
\end{center}

We next consider the domain-wall fermion Dirac operator
\begin{equation}
  (c \otimes \partial_u + \epsilon \otimes \DiracOperatorOnY) + \gradingOperator \otimes (m \signfunction_0) \colon C^{\infty}_c(\R \times Y; \C^2 \otimes E) \to L^2(\R \times Y; \C^2 \otimes E)
\end{equation}
on $\R \times Y$, which is also essentially self-adjoint on $L^2(\R \times Y; \C^2 \otimes E)$.
Assume $\DiracOperatorOnY$ has no zero eigenvalues.
Let $\firstPositiveEigenvalueOfY$ be the positive square root of the first non-zero eigenvalue of $\DiracOperatorOnY^2$.
By the method of separation of variables~\cite{MR751959}*{Theorem VIII.33}, we have
\begin{equation}
  \spectrum \big[ (c \otimes \partial_u + \epsilon \otimes \DiracOperatorOnY) + \gradingOperator \otimes (m \signfunction_0) \big] \cap (-\firstPositiveEigenvalueOfY, \firstPositiveEigenvalueOfY) = \emptyset
\end{equation}
for any $m > \firstPositiveEigenvalueOfY$.
We now proceed to the domain-wall fermion Dirac operator on $X$ via Proposition~\ref{proposition: excision}.
Recall that $D = c \otimes \partial_u + \epsilon \otimes \DiracOperatorOnY$ on the neck $(-4,4) \times Y \subset X$.
\begin{proposition}\label{proposition: spectral gap of DWFD operator}
  Assume $\DiracOperatorOnY$ has no zero eigenvalues.
  Let $\firstPositiveEigenvalueOfY$ be the positive square root of the first non-zero eigenvalue of $\DiracOperatorOnY^2$.
  Then, there exists a constant $m_1 > 0$ that depends only on $\firstPositiveEigenvalueOfY$ such that we have
  \begin{equation}
    \spectrum (D + m \kappa \gradingOperator_S) \cap \left( -\frac{\firstPositiveEigenvalueOfY}{2}, \frac{\firstPositiveEigenvalueOfY}{2} \right) = \emptyset
  \end{equation}
  for any $m > m_1$.
\end{proposition}
\begin{proof}
  We apply Proposition~\ref{proposition: excision} for $\big((c \otimes \partial_u + \epsilon \otimes \DiracOperatorOnY) + \gradingOperator \otimes (m \signfunction_0)\big)$ on $\R \times Y$ and $(D + m \kappa \gradingOperator_S)$ on $X$ with $U_0 = U'_0 = (-4,4) \times Y$.
  Let $\Lambda_2 := \firstPositiveEigenvalueOfY$, $\Lambda_1 := \firstPositiveEigenvalueOfY/2$, and $\Lambda_0 := 0$, and we have $C > 0$ of Proposition~\ref{proposition: excision}.
  We set
  \begin{equation}
    m_1^2 := \max \left\{ \frac{(C^2+\Lambda_2^2)(C^2+\Lambda_1^2)}{\Lambda_2^2-\Lambda_1^2}, \frac{(C^2+\Lambda_1^2)(C^2+\Lambda_0^2)}{\Lambda_1^2-\Lambda_0^2}, (C^2+\Lambda_2^2) \right\},
  \end{equation}
  which yields the conclusion.
\end{proof}

\subsection{Product formula}

Next, we will modify $D$ on $X_+$.
Let $\XPlusHat := (-\infty,0] \times Y \cup X_+$ with the standard cylindrical-end metric.
The bundle $S$ and the operator $D$ naturally extends to $\XPlusHat$, which will be denoted by $\SonXPlusHat$ and $\DiracOperatorOnXPlusHat$.
\begin{center}
  \begin{tikzpicture}
    \draw pic {infiniteHalfCylinderY} pic {XPlus};
  \end{tikzpicture}
\end{center}
Recall~\cite{MR0397797}*{Corollary (3.14)} that $\DiracOperatorOnXPlusHat \colon L^2(\XPlusHat; \SonXPlusHat) \to L^2(\XPlusHat; \SonXPlusHat)$ is a Fredholm operator if $\DiracOperatorOnY$ has no zero eigenvalues; thus, there exists $\firstPositiveEigenvalueOfXPlusHat > 0$ such that $\spectrum \DiracOperatorOnXPlusHat \cap (-\firstPositiveEigenvalueOfXPlusHat, \firstPositiveEigenvalueOfXPlusHat) = \{0\}$.

Let $\signfunction \colon \R \times \XPlusHat \to [-1,1]$ be an $L^{\infty}$-function such that $\signfunction \equiv -1$ on $(-\infty,0) \times \XPlusHat$ and $\signfunction \equiv 1$ on  $(0,\infty) \times \XPlusHat$.
We consider a bundle $\C^2 \otimes \SonXPlusHat$ on $\R \times \XPlusHat$ equipped with a $\Z_2$-grading operator $\gradingOperator \otimes \id_S$ and an odd operator
\begin{equation}
  \epsilon \otimes (\DiracOperatorOnXPlusHat + m\signfunction \gradingOperator_S) + c \otimes \partial_t=
  \begin{pmatrix}
    0 & (\DiracOperatorOnXPlusHat + m\signfunction\gradingOperator_S) + \partial_t \\
    (\DiracOperatorOnXPlusHat + m\signfunction\gradingOperator_S) - \partial_t & 0
  \end{pmatrix},
\end{equation}
which is self-adjoint on $L^2(\R \times \XPlusHat; \C^2 \otimes \SonXPlusHat)$.
Note that this operator is a coordinate change of the graded tensor product of $(c \partial_t + m\signfunction_0\epsilon)$ and $\DiracOperatorOnXPlusHat$.
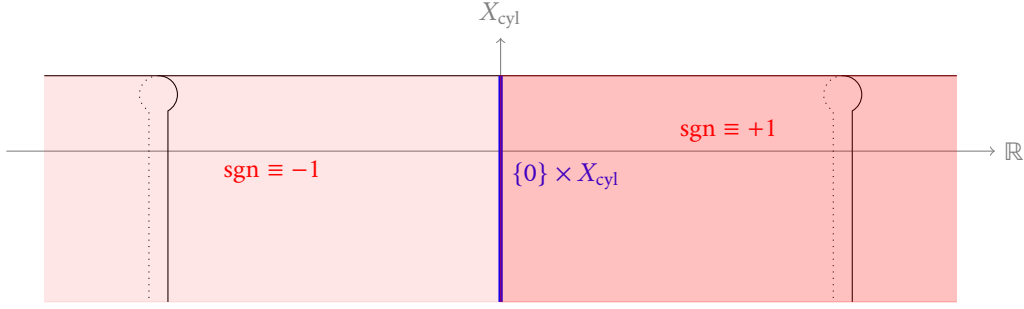
\begin{center}
  \begin{tikzpicture}
    \pic {infiniteCylinderXPlusHat};
    \draw [blue]
      [ultra thick] (0,-2) -- (0,1)
      node at (0,0) [below right]{$\{0\} \times \XPlusHat$};
    \filldraw [red]
      [very nearly transparent] (-6,1) -- (0,1) -- (0,-2) -- (-6,-2);
    \filldraw [red]
      [nearly transparent] (6,1) -- (0,1) -- (0,-2) -- (6,-2);
    \draw [red]
      node at (-3,0) [below]{$\signfunction \equiv -1$}
      node at (3,0) [above]{$\signfunction \equiv +1$};
  \end{tikzpicture}
\end{center}

\begin{proposition}\label{proposition: product formula}
  If $\DiracOperatorOnY$ has no zero eigenvalues, then the operator $\big( \epsilon \otimes (\DiracOperatorOnXPlusHat + m\signfunction \gradingOperator_S) + c \otimes \partial_t \big)$ is also Fredholm, and we have
  \begin{equation}
    \ASindex(\DiracOperatorOnXPlusHat) = - \ASindex \big[\epsilon \otimes (\DiracOperatorOnXPlusHat + m\signfunction \gradingOperator_S) + c \otimes \partial_t \big]
  \end{equation}
  for any $m > 0$, and
  \begin{equation}
    \spectrum \big[ \epsilon \otimes (\DiracOperatorOnXPlusHat + m\signfunction \gradingOperator_S) + c \otimes \partial_t \big] \cap (-\firstPositiveEigenvalueOfXPlusHat, \firstPositiveEigenvalueOfXPlusHat) = \{0\}
  \end{equation}
  for any $m > \firstPositiveEigenvalueOfXPlusHat$.
\end{proposition}
\begin{proof}
  Assume $\DiracOperatorOnXPlusHat \phi = 0$.
  Set $\phi_{\pm} := (\phi \pm \gradingOperator_S \phi)/2$.
  Recall that $\big(e^{-m\abs{t}}\big)' = -m \signfunction_0 e^{-m\abs{t}}$ for any $m > 0$.
  Then, we have
  \begin{equation}
    \begin{pmatrix}
    0 & (\DiracOperatorOnXPlusHat + m\signfunction\gradingOperator_S) + \partial_t \\
    (\DiracOperatorOnXPlusHat + m\signfunction\gradingOperator_S) - \partial_t & 0
  \end{pmatrix}
  \begin{pmatrix}
    e^{-m\abs{t}} \phi_- \\
    e^{-m\abs{t}} \phi_+
  \end{pmatrix}
  = 0.
  \end{equation}
  The details are left to the reader.
\end{proof}

\subsection{Embeddings into a cylinder}\label{subsection: Embeddings}

At the heart of this paper lies our next step, which constructs an embedding $\tau$ from $(-2,2) \times \XPlusHat$ into the infinite cylinder $\R \times X$.
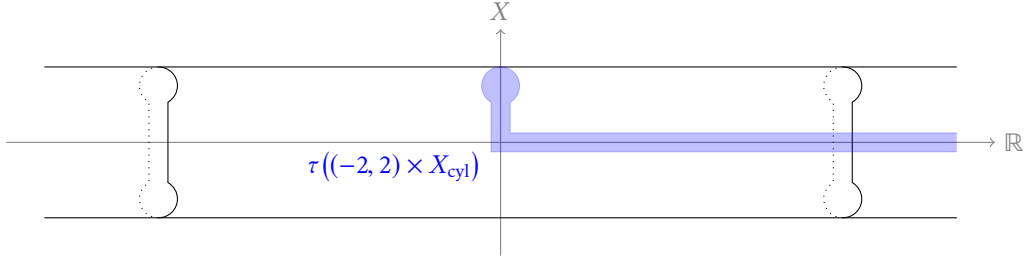
\begin{center}
  \begin{tikzpicture}
    \pic {infiniteCylinderX};
    \filldraw [blue, nearly transparent]
      (6,0.125) -- (0.125,0.125) -- (0.125,0.5335) arc [start angle=-60, end angle=240, radius=0.25] -- (-0.125,-0.125) -- (6,-0.125);
    \draw [blue]
      node at (-0.125,0) [below left]{$\tau \big( (-2,2) \times \XPlusHat \big)$};
  \end{tikzpicture}
\end{center}

Let $R_1 := (-2,2) \times (-\infty, 4)$ and $R_2 := \R \times (-4,4)$.
We denote the coordinates of $R_1$ by $(t,u)$ and that of $R_2$ by $(s,v)$.
Fix an embedding $\tau_{\R^2} \colon R_1 \to R_2$ such that $\tau_{\R^2} \equiv \id$ for $2 \le u$ and
\begin{equation}
  \begin{pmatrix}
    t \\
    u
  \end{pmatrix}
  \mapsto
  \begin{pmatrix}
    -u \\
    t
  \end{pmatrix}
\end{equation}
for $u \le -100$.
Since $Y$ has a collar neighbourhood isometric to $(-4,4) \times Y$, we can regard $R_1 \times Y$ and $R_2 \times Y$ as open subsets of $(-2,2) \times \XPlusHat$ and $\R \times X$ respectively.
We define an embedding
\begin{equation}
  \tau \colon (-2,2) \times \XPlusHat \to \R \times X
\end{equation}
by $\tau \equiv \id_{\R} \times \id_X$ on $(-2,2) \times X_+$ and $\tau \equiv \tau_{\R^2} \times \id_Y$ on $R_1 \times Y$.
By construction, $\tau$ is an isometry outside a compact set $\big( (-2,2) \times (-100, 2) \times Y \big)$.
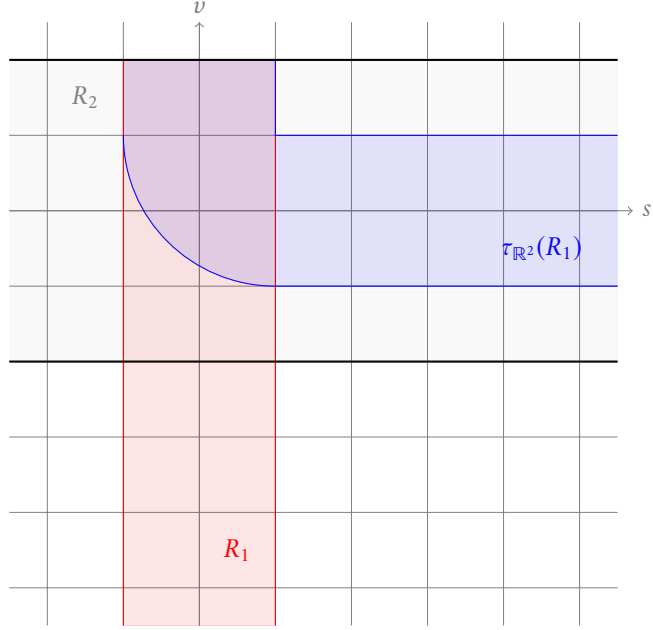
\begin{center}
  \begin{tikzpicture}
  % Coordinates
    \draw [gray, very thin] (-2.5,-5.5) grid (5.5,2.5);
    \draw [gray, ->] (-2.5,0) -- (5.7,0) node [right]{$s$};
    \draw [gray, ->] (0,-5.5) -- (0,2.5) node [above]{$v$};
  % R_1
    \draw [red]
      (-1,2) -- (-1,-5.5)
      (1,2) -- (1,-5.5);
    \filldraw [red] [very nearly transparent] (-1,2) -- (1,2) -- (1,-5.5) -- (-1,-5.5);
    \node [red] at (0.5,-4.5) {$R_1$};
  % embedded R_1
    \draw [blue]
      (5.5,1) -- (1,1) -- (1,2)
      (-1,1) arc [start angle = 180, end angle=270, radius=2] -- (5.5,-1);
    \filldraw [blue] [very nearly transparent] (5.5,1) -- (1,1) -- (1,2) -- (-1,2) -- (-1,1) arc [start angle = 180, end angle=270, radius=2] -- (5.5,-1);
    \node [blue] at (4.5,-0.5) {$\tau_{\R^2}(R_1)$};
  % R_2
    \draw [thick]
      (-2.5,2) -- (5.5,2)
      (-2.5,-2) -- (5.5,-2);
    \filldraw [gray] [ultra nearly transparent] (-2.5,2) -- (5.5,2) -- (5.5,-2) -- (-2.5,-2);
    \node [gray] at (-1.5,1.5) {$R_2$};
  \end{tikzpicture}
\end{center}

We modify the Riemannian metric on $\R \times X$ so that $\tau$ becomes an isometry.
Let $g$ denote the product metric on $(-2,2) \times \XPlusHat$ and $g'$ on $\R \times X$.
Let $\chi \colon \R \times X \to [0,1]$ be a bump function such that $\chi \equiv 1$ on $\tau \big( (-1,1) \times \XPlusHat \big)$ and $\chi \equiv 0$ outside $\tau \big( (-2,2) \times \XPlusHat \big)$.
We define a family of Riemannian metrics $g'_r$ connecting $g'_0 := g'$ to $g'_1 := \chi \big( (\tau^{-1})^*g \big) + (1-\chi)g'$ by
\begin{equation}
  g'_r := (1-r) g' + r \big( \chi \big( (\tau^{-1})^*g \big) + (1-\chi)g' \big)
\end{equation}
for $r \in [0,1]$.
Now $\tau$ is an isometry from $\big( (-1,1) \times \XPlusHat, g \big)$ to $\big( \R \times X, g'_1 \big)$.
We also modify the hermitian metric on $S$ conformally so that the $L^2$-norm on $C^{\infty}_c(\R \times X; \C^2 \otimes S)$ remains unchanged.
Let $d\mu(g'_r)$ denote the volume form of $\big( \R \times X, g'_r \big)$, and we define $f_r \colon \R \times X \to \R$ by $d\mu(g'_r) = e^{2f_r}d\mu(g')$.
Let $\innerProduct{\cdot}{\cdot}$ denote the hermitian metric on $S$.
We define a family of hermitian metrics $\innerProduct{\cdot}{\cdot}_r$ on $S$ by
\begin{equation}
  \innerProduct{\cdot}{\cdot}_r := e^{-2f_r} \innerProduct{\cdot}{\cdot}
\end{equation}
for $r \in [0,1]$.
Now the $L^2$-norm on $C^{\infty}_c(\R \times X; \C^2 \otimes S)$ with $d\mu(g'_r)$ and $\innerProduct{\cdot}{\cdot}_r$ remains unchanged.
Using $\Spin(2)$-action on $\C^2$, we can lift $\tau$ to $\tilde{\tau} \colon \C^2 \otimes \SonXPlusHat \to \C^2 \otimes S$ so that
\begin{equation}\label{equation: coordinate change}
  \tilde{\tau}^{-1} \circ \big[ \epsilon \otimes D + c \otimes \partial_s \big] \circ \tilde{\tau} = \big[ \epsilon \otimes \DiracOperatorOnXPlusHat + c \otimes \partial_t \big]
\end{equation}
holds.
Note that $g'_r$ coincides with $g'$ outside a compact set.
We refer the reader to~\cite{MR1158762} for a related construction.

\subsection{The eta invariant of domain-wall fermion Dirac operators}

We will slightly extend the definition of the eta invariant to take care of discontinuities of domain-wall fermion Dirac operators.

Fix $m > 0$.
Since $\Ker (D + m \kappa \gradingOperator_S) = \{0\}$, there exists a constant $C_m > 0$ such that
\begin{equation}
  \Ker (D + m \kappa \gradingOperator_S + f) = \{0\}
\end{equation}
for any $f \in L^2(Z; \Endo(S_Z))$ with $\norm{f}_2 < C_m$.

\begin{lemma}\label{lemma: eta invariant}
  Let $f_1 \in L^2(Z; \Endo(S_Z))$ with $\norm{f_1}_2 < C_m$ and $f_2 \in L^2(Z; \Endo(S_Z))$ with $\norm{f_2}_2 < C_m$.
  Assume that $m \kappa \gradingOperator_S + f_1$ and $m \kappa \gradingOperator_S + f_2$ are smooth operators.
  Then, we have
  \begin{equation}
    \eta(D + m \kappa \gradingOperator_S + f_1) = \eta(D + m \kappa \gradingOperator_S + f_2).
  \end{equation}
\end{lemma}
\begin{proof}
  This is a direct consequence of the variational formula of the eta invariant~\cite{MR1396308}*{Theorem 1.13.2}.
\end{proof}

We now define the eta invariant of domain-wall fermion Dirac operators by
\begin{equation}\label{definition: eta invariant of domain-wall fermion Dirac operators}
  \eta (D + m \kappa \gradingOperator_S) := \eta (D + m \kappa \gradingOperator_S + f)
\end{equation}
for any $f \in L^2(Z; \Endo(S_Z))$ such that $\norm{f}_2 < C_m$ and that $m \kappa \gradingOperator_S + f$ is a smooth operator.
This definition is well defined by Lemma~\ref{lemma: eta invariant}.

\subsection{Main theorem}

\begin{theorem}\label{theorem: main theorem}
  If $\DiracOperatorOnY \colon C^{\infty}(Y;E) \to C^{\infty}(Y;E)$ has no zero eigenvalues, then there exists a constant $m_0 > 0$ that depends only on $X$, $S$, and $D$ such that we have
  \begin{equation}
    \APSindex(\restrictedTo{D}{X_+}) = \frac{\eta(D + m\kappa\gradingOperator_S) - \eta(D - m \gradingOperator_S)}{2}
  \end{equation}
  for any $m > m_0$.
\end{theorem}

\begin{proof}
We begin with the observation that $\R \times X \setminus \tau \big( \{0\} \times \XPlusHat \big)$ has two connected components; we will denote by $(\R \times X)_-$ the one containing $\{-10\} \times X_+$ and by $(\R \times X)_+$ the other half.
Let $\kappaAPS \colon \R \times X \to [-1,1]$ be an $L^{\infty}$-function such that $\kappaAPS \equiv \pm 1$ on $(\R \times X)_{\pm}$.
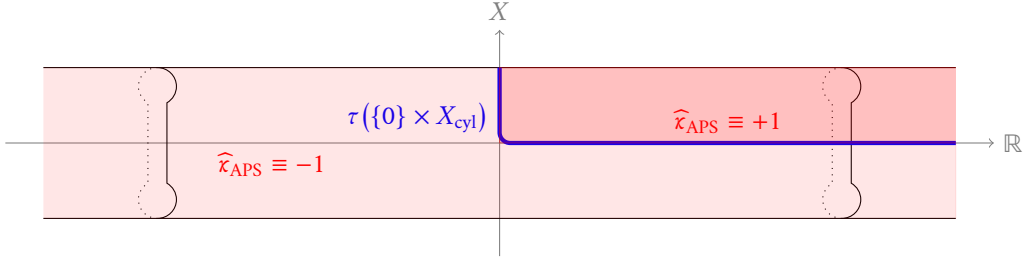
\begin{center}
  \begin{tikzpicture}
    \pic {infiniteCylinderX};
    \draw [blue]
      [ultra thick, rounded corners] (0,1) -- (0,0) -- (6,0)
      node at (0,0) [above left]{$\tau \big( \{0\} \times \XPlusHat \big)$};
    \filldraw [red]
      [very nearly transparent] (-6,1) -- (0,1) -- (0,0) -- (6,0) -- (6,-1) -- (-6,-1);
    \filldraw [red]
      [nearly transparent] (6,1) -- (0,1) -- (0,0) -- (6,0);
    \draw [red]
      node at (-3,0) [below]{$\kappaAPS \equiv -1$}
      node at (3,0) [red, above]{$\kappaAPS \equiv +1$};
  \end{tikzpicture}
\end{center}

Fix $m > 0$.
We introduce an operator $\DiracOperatorOnCylinderAPS_m \colon C^{\infty}_c(\R \times X; \C^2 \otimes S) \to L^2(\R \times X; \C^2 \otimes S)$ defined by
\begin{equation}\label{equation: main operator}
  \DiracOperatorOnCylinderAPS_m := \epsilon \otimes (D + m \kappaAPS \gradingOperator_S) + c \otimes \partial_s =
  \begin{pmatrix}
    0 & (D + m \kappaAPS \gradingOperator_S) + \partial_s \\
    (D + m \kappaAPS \gradingOperator_S) - \partial_s & 0
  \end{pmatrix},
\end{equation}
which is essentially self-adjoint on $L^2(\R \times X; \C^2 \otimes S)$.

By~\eqref{equation: coordinate change}, we have
\begin{equation}
  \tilde{\tau}^{-1} \circ \DiracOperatorOnCylinderAPS_m \circ \tilde{\tau} = \big[ \epsilon \otimes (\DiracOperatorOnXPlusHat + m\signfunction \gradingOperator_S) + c \otimes \partial_t \big].
\end{equation}
By Proposition~\ref{proposition: product formula}, we have
\begin{equation}
  \spectrum \big[ \epsilon \otimes (\DiracOperatorOnXPlusHat + m\signfunction \gradingOperator_S) + c \otimes \partial_t \big] \cap (-\firstPositiveEigenvalueOfXPlusHat, \firstPositiveEigenvalueOfXPlusHat) = \{0\}
\end{equation}
for any $m > \firstPositiveEigenvalueOfXPlusHat$.
Set $\Lambda := \firstPositiveEigenvalueOfXPlusHat$.
We now apply Theorem~\ref{theorem: localisation} for $\big( \epsilon \otimes (\DiracOperatorOnXPlusHat + m\signfunction \gradingOperator_S) + c \otimes \partial_t \big)$ on $\R \times \XPlusHat$ and $\DiracOperatorOnCylinderAPS_m$ on $\R \times X$ equipped with the modified metric $g'_1$.
By homotopy invariance of the index, we can use either $g'$ or $g'_1$ to compute $\ASindex (\DiracOperatorOnCylinderAPS_m)$.
Thus, we have the constant $C > 0$ such that
\begin{equation}\label{eq: product and embedding}
  \ASindex \big[ \epsilon \otimes (\DiracOperatorOnXPlusHat + m\signfunction \gradingOperator_S) + c \otimes \partial_t \big] = \ASindex (\DiracOperatorOnCylinderAPS_m)
\end{equation}
for $m > 2(C^2+\Lambda^2)/\Lambda$.
Note that $2(C^2+\Lambda^2)/\Lambda > \Lambda = \firstPositiveEigenvalueOfXPlusHat$.
By Proposition~\ref{proposition: product formula}, we also have
\begin{equation}\label{eq: embedding and aps}
  \ASindex (\DiracOperatorOnXPlusHat) = - \ASindex \big[ \epsilon \otimes (\DiracOperatorOnXPlusHat + m\signfunction \gradingOperator_S) + c \otimes \partial_t \big]
\end{equation}
for any $m > 0$.

To apply the Atiyah-Patodi-Singer index theorem, we will approximate $\DiracOperatorOnCylinderAPS_m$ by an operator with smooth coefficients.
Let $C_m > 0$ be the same constant as in Lemma~\ref{lemma: eta invariant}.
Let $\kappaAPSsmooth \colon \R \times X \to [-1,1]$ be a \emph{smooth} function that approximates $\kappaAPS$ such that $\kappaAPSsmooth \equiv -1$ on $\{-10\} \times X$,
\begin{equation}
  \norm[\big]{\restrictedTo{\kappaAPSsmooth}{\{10\} \times X} - \kappa}_{L^2(\{10\} \times X)} < C_m,
\end{equation}
and $\ASindex (\DiracOperatorOnCylinderAPS_m) = \ASindex (\smoothDiracOperatorOnCylinderAPS_m)$, where
\begin{equation}
  \smoothDiracOperatorOnCylinderAPS_m := \epsilon \otimes (D + m \kappaAPSsmooth \gradingOperator_S) + c \otimes \partial_s =
  \begin{pmatrix}
    0 & (D + m \kappaAPSsmooth \gradingOperator_S) + \partial_s \\
    (D + m \kappaAPSsmooth \gradingOperator_S) - \partial_s & 0
  \end{pmatrix}.
\end{equation}
Note that
\begin{equation}
  \DiracOperatorOnCylinderAPS_m =
  \begin{cases}
    \epsilon \otimes (D - m\gradingOperator_S) + c \otimes \partial_s &\text{ on $\{-10\} \times X$} \\
    \epsilon \otimes \big( D + m \superemph{\big(\restrictedTo{\kappaAPSsmooth}{\{10\} \times X}\big)} \gradingOperator_S \big) + c \otimes \partial_s &\text{ on $\{+10\} \times X$}.
  \end{cases}
\end{equation}
By Proposition~\ref{proposition: spectral gap of DWFD operator}, the domain-wall fermion Dirac operator $D + m \kappa \gradingOperator_S$ has no zero eigenvalues if $m > m_1$;hence, neither does $D + m \big(\restrictedTo{\kappaAPSsmooth}{\{10\} \times X}\big) \gradingOperator_S$.
The Atiyah-Patodi-Singer index theorem on the cylinder \cite{MR0397797}*{(2.27)} yields
\begin{equation}
  \ASindex (\smoothDiracOperatorOnCylinderAPS_m) = - \frac{\eta\big( D + m \big(\restrictedTo{\kappaAPSsmooth}{\{10\} \times X}\big) \gradingOperator_S \big) - \eta(D - m \gradingOperator_S)}{2},
\end{equation}
where we have used the assumption that $\dim (\R \times X)$ is odd and the fact~\cite{MR1396308}*{Lemma 1.8.2 (d)} that the constant term in the asymptotic expansion of the heat kernel of an elliptic \emph{differential} operator on an odd-dimensional manifold vanishes\footnote{This does not remain true for pseudodifferential operators~\cite{MR1031992}*{Theorem 13.12}.}.
By the definition~\eqref{definition: eta invariant of domain-wall fermion Dirac operators}, we have $\eta\big( D + m \big(\restrictedTo{\kappaAPSsmooth}{\{10\} \times X}\big) \gradingOperator_S \big) = \eta(D + m \kappa \gradingOperator_S)$.
By assumption, we have $\ASindex (\DiracOperatorOnCylinderAPS_m) = \ASindex (\smoothDiracOperatorOnCylinderAPS_m)$.
Thus, we get
\begin{equation}\label{eq: product and eta}
  \ASindex (\DiracOperatorOnCylinderAPS_m) = - \frac{\eta(D + m \superemph{\kappa} \gradingOperator_S) - \eta(D - m \gradingOperator_S)}{2}.
\end{equation}

Now set $m_0 := \max \{ m_1, 2(C^2+\Lambda^2)/\Lambda \}$.
Note that $\APSindex (\restrictedTo{D}{X_+}) = \ASindex (\DiracOperatorOnXPlusHat)$ by \cite{MR0397797}*{Proposition 3.11}.
Combining \eqref{eq: product and embedding}, \eqref{eq: embedding and aps}, and \eqref{eq: product and eta}, we have
\begin{equation}
  \APSindex (\restrictedTo{D}{X_+}) = \frac{\eta(D + m \kappa \gradingOperator_S) - \eta(D - m \gradingOperator_S)}{2}
\end{equation}
for $m > m_0$.
The proof is complete.
\end{proof}

\begin{remark}
  The particular choice of $\kappa$ will not be essential in our arguments but clarify our idea.
  See Lemma~\ref{lemma: eta invariant}.
  Our arguments extend easily to the case when $\kappa$ satisfies $\kappa \equiv \kappa_{\pm}$ on $X_{\pm} \setminus (-2,2) \times Y$ for some $\kappa_{\pm} \in \R$ with $\kappa_+ \kappa_- < 0$.
  In particular, taking an extreme limit $(-\kappa_-) \gg \kappa_+$, we recover Shamir domain-wall fermions~\cites{Shamir, Furman-Shamir}.
  See~\cite{MR3873281}*{IV.B}.
\end{remark}

\begin{remark}
  We have used the Atiyah-Patodi-Singer theorem only on \emph{cylinders}.
\end{remark}

We conclude this paper with a problem.
Although the proof above implies that there are no \emph{edge-of-edge states} or \emph{corner states} in our situation, we expect that corner states will emerge if we introduce extra domain walls.
Let $\kappaCorner \colon \R \times X \to [-1,1]$ be a bump function such that $\kappaCorner \equiv 1$ on $[-10,10] \times X$ and $\kappaCorner \equiv -1$ outside $[-10,10] \times X$.
Fix $M \gg 0$.
We consider yet another self-adjoint operator $\DiracOperatorOnCylinderAPS_m + M\kappaCorner(\gradingOperator \otimes \id_S)$ so that corner states would emerge around $\{10\} \times Y$.
Then, perturbation arguments as in~\cite{MR3873281} lead us to ask whether
\begin{equation}
  \APSindex \big(\DiracOperatorOnCylinderAPS_m \text{ on } [-10,10] \times X \big) = \frac{\eta \big( \DiracOperatorOnCylinderAPS_m + M\kappaCorner(\gradingOperator \otimes \id_S) \big) - \eta \big( \DiracOperatorOnCylinderAPS_m - M(\gradingOperator \otimes \id_S) \big)}{2}
\end{equation}
holds with some regularisation to define the right-hand side.

%%%%%%%%%%%%%%%%%%%%%%%%%%%%%%%%%%%%%%%%%%%%%%%%%%%%%%%%%%%%%%%%%%%%%%%%%%%%%%%%
\section*{Acknowledgements}
%%%%%%%%%%%%%%%%%%%%%%%%%%%%%%%%%%%%%%%%%%%%%%%%%%%%%%%%%%%%%%%%%%%%%%%%%%%%%%%%

The authors wish to express their gratitude to the organisers of the workshop \textit{Progress in the Mathematics of Topological States of Matter}, which triggered our collaboration.
The authors also wish to express their thanks for helpful discussions during the preparation of this paper to S. Aoki, Y. Hamada, M. Hamanaka, K. Hashimoto, S. Hayashi, N. Kawai, Y. Kikukawa, T. Kimura, Y. Kubota, Y. Matsuki, T. Misumi, M. Mori, H. Moriyoshi, K. Nakayama, T. Natsume, H. Suzuki, and K. Yonekura.
The authors are also grateful to the anonymous referees for carefully reading the paper and making many useful suggestions.

Hidenori Fukaya is supported in part by JSPS KAKENHI Grant Numbers JP18H01216 and JP18H04484.
Mikio Furuta is supported in part by JSPS KAKENHI Grant Number JP17H06461.
Shinichiroh Matsuo is supported in part by JSPS KAKENHI Grant Number JP17K14186.
Tetsuya Onogi is supported in part by JSPS KAKENHI Grant Number JP18K03620.
Satoshi Yamaguchi is supported in part by JSPS KAKENHI Grant Number JP15K05054.
Mayuko Yamashita is supported in part by JSPS KAKENHI Grant Number 19J22404.

%%%%%%%%%%%%%%%%%%%%%%%%%%%%%%%%%%%%%%%%%%%%%%%%%%%%%%%%%%%%%%%%%%%%%%%%%%%%%%%%


\begin{bibdiv}
\begin{biblist}

\bib{MR0397797}{article}{
   author={Atiyah, M. F.},
   author={Patodi, V. K.},
   author={Singer, I. M.},
   title={Spectral asymmetry and Riemannian geometry. I},
   journal={Math. Proc. Cambridge Philos. Soc.},
   volume={77},
   date={1975},
   pages={43--69},
   issn={0305-0041},
   review={\MR{0397797}},
   doi={10.1017/S0305004100049410},
}

\bib{MR0397798}{article}{
   author={Atiyah, M. F.},
   author={Patodi, V. K.},
   author={Singer, I. M.},
   title={Spectral asymmetry and Riemannian geometry. II},
   journal={Math. Proc. Cambridge Philos. Soc.},
   volume={78},
   date={1975},
   number={3},
   pages={405--432},
   issn={0305-0041},
   review={\MR{0397798}},
   doi={10.1017/S0305004100051872},
}

\bib{MR0397799}{article}{
   author={Atiyah, M. F.},
   author={Patodi, V. K.},
   author={Singer, I. M.},
   title={Spectral asymmetry and Riemannian geometry. III},
   journal={Math. Proc. Cambridge Philos. Soc.},
   volume={79},
   date={1976},
   number={1},
   pages={71--99},
   issn={0305-0041},
   review={\MR{0397799}},
   doi={10.1017/S0305004100052105},
}

\bib{MR1158762}{article}{
   author={Bourguignon, Jean-Pierre},
   author={Gauduchon, Paul},
   title={Spineurs, opérateurs de Dirac et variations de métriques},
   language={French, with English summary},
   journal={Comm. Math. Phys.},
   volume={144},
   date={1992},
   number={3},
   pages={581--599},
   issn={0010-3616},
   review={\MR{1158762}},
   doi={10.1007/BF02099184},
}

\bib{MR779916}{article}{
   author={Callan, C. G., Jr.},
   author={Harvey, J. A.},
   title={Anomalies and fermion zero modes on strings and domain walls},
   journal={Nuclear Phys. B},
   volume={250},
   date={1985},
   number={3},
   pages={427--436},
   issn={0550-3213},
   review={\MR{779916}},
   doi={10.1016/0550-3213(85)90489-4},
}

\bib{MR3633291}{article}{
   author={Fefferman, C. L.},
   author={Lee-Thorp, J. P.},
   author={Weinstein, M. I.},
   title={Topologically protected states in one-dimensional systems},
   journal={Mem. Amer. Math. Soc.},
   volume={247},
   date={2017},
   number={1173},
   pages={vii+118},
   issn={0065-9266},
   isbn={978-1-4704-2323-0},
   isbn={978-1-4704-3707-7},
   review={\MR{3633291}},
   doi={10.1090/memo/1173},
}

\bib{PhysRevLett.42.1195}{article}{
  author={Fujikawa, Kazuo},
  title={Path-Integral Measure for Gauge-Invariant Fermion Theories},
  journal={Phys. Rev. Lett.},
  volume={42},
  pages={1195--1198},
  year={1979},
  publisher={American Physical Society},
  doi={10.1103/PhysRevLett.42.1195},
  url={https://link.aps.org/doi/10.1103/PhysRevLett.42.1195},
}

\bib{MR3873281}{article}{
  author={Fukaya, Hidenori},
  author={Onogi, Tetsuya},
  author={Yamaguchi, Satoshi},
  title={Atiyah-Patodi-Singer index from the domain-wall fermion Dirac
  operator},
  journal={Phys. Rev. D},
  volume={96},
  date={2017},
  number={12},
  pages={125004, 22},
  issn={2470-0010},
  review={\MR{3873281}},
  doi={10.1103/physrevd.96.125004},
}

\bib{Furman-Shamir}{article}{
  author={Furman, Vadim},
  author={Shamir, Yigal},
  title={Axial symmetries in lattice QCD with Kaplan fermions},
  journal={Nucl. Phys.},
  volume={B439},
  year={1995},
  pages={54--78},
  doi={10.1016/0550-3213(95)00031-M},
  eprint={hep-lat/9405004},
}

\bib{MR2361481}{book}{
   author={Furuta, Mikio},
   title={Index theorem. 1},
   series={Translations of Mathematical Monographs},
   volume={235},
   note={Translated from the 1999 Japanese original by Kaoru Ono;
   Iwanami Series in Modern Mathematics},
   publisher={American Mathematical Society, Providence, RI},
   date={2007},
   pages={xviii+205},
   isbn={978-0-8218-2097-1},
   review={\MR{2361481}},
}

\bib{MR1396308}{book}{
   author={Gilkey, Peter B.},
   title={Invariance theory, the heat equation, and the Atiyah-Singer index
   theorem},
   series={Studies in Advanced Mathematics},
   edition={2},
   publisher={CRC Press, Boca Raton, FL},
   date={1995},
   pages={x+516},
   isbn={0-8493-7874-4},
   review={\MR{1396308}},
}

\bib{Gromov:2015fda}{article}{
  author={Gromov, Andrey},
  author={Jensen, Kristan},
  author={Abanov, Alexander G.},
  title={Boundary effective action for quantum Hall states},
  journal={Phys. Rev. Lett.},
  volume={116},
  year={2016},
  number={12},
  pages={126802},
  doi={10.1103/PhysRevLett.116.126802},
}

\bib{MR0468803}{article}{
   author={Jackiw, R.},
   author={Rebbi, C.},
   title={Solitons with fermion number $1/2$},
   journal={Phys. Rev. D (3)},
   volume={13},
   date={1976},
   number={12},
   pages={3398--3409},
   issn={0556-2821},
   review={\MR{0468803}},
   doi={10.1103/PhysRevD.13.3398},
}

\bib{Kaplan}{article}{
   author={Kaplan, David B.},
   title={A method for simulating chiral fermions on the lattice},
   journal={Phys. Lett. B},
   volume={288},
   date={1992},
   number={3-4},
   pages={342--347},
   issn={0370-2693},
   review={\MR{1179379}},
   doi={10.1016/0370-2693(92)91112-M},
}

\bib{Kitaev}{article}{
  author={Kitaev, Alexei},
  title={Periodic table for topological insulators and superconductors},
  journal={AIP Conference Proceedings},
  volume={1134},
  number={1},
  pages={22-30},
  year={2009},
  doi={10.1063/1.3149495},
  eprint={https://aip.scitation.org/doi/pdf/10.1063/1.3149495},
}

\bib{MR1031992}{book}{
   author={Lawson, H. Blaine, Jr.},
   author={Michelsohn, Marie-Louise},
   title={Spin geometry},
   series={Princeton Mathematical Series},
   volume={38},
   publisher={Princeton University Press, Princeton, NJ},
   date={1989},
   pages={xii+427},
   isbn={0-691-08542-0},
   review={\MR{1031992}},
}

\bib{MR751959}{book}{
   author={Reed, Michael},
   author={Simon, Barry},
   title={Methods of modern mathematical physics. I},
   edition={2},
   note={Functional analysis},
   publisher={Academic Press, Inc. [Harcourt Brace Jovanovich, Publishers], New York},
   date={1980},
   pages={xv+400},
   isbn={0-12-585050-6},
   review={\MR{751959}},
}

\bib{MR3628684}{article}{
   author={Seiberg, Nathan},
   author={Witten, Edward},
   title={Gapped boundary phases of topological insulators via weak
   coupling},
   journal={PTEP. Prog. Theor. Exp. Phys.},
   date={2016},
   number={12},
   pages={12C101, 78},
   issn={2050-3911},
   review={\MR{3628684}},
   doi={10.1093/ptep/ptw083},
}

\bib{Shamir}{article}{
  author={Shamir, Yigal},
  title={Chiral fermions from lattice boundaries},
  journal={Nucl. Phys.},
  volume={B406},
  year={1993},
  pages={90--106},
  doi={10.1016/0550-3213(93)90162-I},
}

\bib{MR3565832}{article}{
  author={Tachikawa, Yuji},
  author={Yonekura, Kazuya},
  title={Gauge interactions and topological phases of matter},
  journal={PTEP. Prog. Theor. Exp. Phys.},
  date={2016},
  number={9},
  pages={093B07, 51},
  issn={2050-3911},
  review={\MR{3565832}},
}

\bib{MR3858615}{article}{
   author={Vassilevich, Dmitri},
   title={Index theorems and domain walls},
   journal={J. High Energy Phys.},
   date={2018},
   number={7},
   pages={108, front matter + 12},
   issn={1126-6708},
   review={\MR{3858615}},
   doi={10.1007/jhep07(2018)108},
}

\bib{MR683171}{article}{
  author={Witten, Edward},
  title={Supersymmetry and Morse theory},
  journal={J. Differential Geom.},
  volume={17},
  date={1982},
  number={4},
  pages={661--692 (1983)},
  issn={0022-040X},
  review={\MR{683171}},
  doi={10.4310/jdg/1214437492},
}

\bib{RevModPhys.88.035001}{article}{
  author={Witten, Edward},
  title={Fermion path integrals and topological phases},
  journal={Rev. Mod. Phys.},
  volume={88},
  pages={035001},
  year={2016},
  publisher={American Physical Society},
  doi={10.1103/RevModPhys.88.035001},
}

\bib{Witten-Yonekura}{article}{
  author={Witten, Edward},
  author={Yonekura, Kazuya},
  title={Anomaly Inflow and the $\eta$-Invariant},
  year={2019},
  eprint={https://arxiv.org/abs/1909.08775},
}

\bib{MR3557925}{article}{
  author={Yonekura, Kazuya},
  title={Dai-Freed theorem and topological phases of matter},
  journal={J. High Energy Phys.},
  date={2016},
  number={9},
  pages={022, front matter+33},
  issn={1126-6708},
  review={\MR{3557925}},
  doi={10.1007/JHEP09(2016)022},
}

\bib{MR3951702}{article}{
  author={Yonekura, Kazuya},
  title={On the cobordism classification of symmetry protected topological
  phases},
  journal={Comm. Math. Phys.},
  volume={368},
  date={2019},
  number={3},
  pages={1121--1173},
  issn={0010-3616},
  review={\MR{3951702}},
  doi={10.1007/s00220-019-03439-y},
}

\bib{MR3611419}{article}{
   author={Yu, Yue},
   author={Wu, Yong-Shi},
   author={Xie, Xincheng},
   title={Bulk-edge correspondence, spectral flow and Atiyah-Patodi-Singer theorem for the $\mathcal{Z}_2$-invariant in topological insulators},
   journal={Nuclear Phys. B},
   volume={916},
   date={2017},
   pages={550--566},
   issn={0550-3213},
   review={\MR{3611419}},
   doi={10.1016/j.nuclphysb.2017.01.018},
}

\end{biblist}
\end{bibdiv}
\end{document}